\documentclass[a4paper,headings = small, abstract]{scrartcl}
\usepackage[english]{babel}
\usepackage{amsmath}
\usepackage{amscd,amsthm,array,bbm,hhline,mathrsfs, enumerate,dsfont, booktabs,fancybox,calc,textcomp,xcolor,graphicx,bbm,xspace,nicefrac,stmaryrd,url,arcs,listings,nameref,subfig,float}

\makeatletter
\newcommand{\leqnomode}{\tagsleft@true}
\newcommand{\reqnomode}{\tagsleft@false}
\makeatother
\usepackage[theoremfont]{newpxtext}
\usepackage[vvarbb, upint, bigdelims]{newpxmath}
\usepackage[scr=boondoxupr, scrscaled = 1.05, cal = euler, calscaled =1.05]{mathalpha}
\usepackage[scaled=0.95]{inconsolata}
\DeclareMathAlphabet{\mathsf}{OT1}{\sfdefault}{m}{n}
\SetMathAlphabet{\mathsf}{bold}{OT1}{\sfdefault}{b}{n}
\usepackage[utf8]{inputenc}

\usepackage[shortlabels]{enumitem}
\usepackage{etoolbox}
\usepackage[normalem]{ulem}
\usepackage{mathtools}
\usepackage{geometry}
\usepackage{float}

\usepackage{bm}
\usepackage{tikz}
\usepackage[outdir =./]{epstopdf}
\usepackage[citestyle = numeric-comp,bibstyle = numeric, firstinits=true, natbib = true, backend = bibtex,maxbibnames=99,url=false, doi=false]{biblatex}
\usepackage{authblk}
\numberwithin{equation}{section}
\graphicspath{{./simulationen/}}
\usepackage{chngcntr}
\counterwithin{figure}{section}

\usepackage{xcolor}
\usepackage[colorlinks=true, allcolors=myteal]{hyperref}
\definecolor{WIMgreen}{RGB}{60 134 132}
\definecolor{red_pers}{RGB}{204 37 41}
\definecolor{UMblue}{RGB}{4 47 86}
\definecolor{myteal}{RGB}{0 123 137}
\definecolor{nd}{RGB}{0 0 0}
\definecolor{dartmouthgreen}{rgb}{0.05, 0.5, 0.06}\definecolor{cobalt}{rgb}{0.0, 0.28, 0.67}\definecolor{coolblack}{rgb}{0.0, 0.18, 0.39}
\definecolor{glaucous}{rgb}{0.38, 0.51, 0.71}\definecolor{hooker\'sgreen}{rgb}{0.0, 0.44, 0.0}\definecolor{lemonchiffon}{rgb}{1.0, 0.98, 0.8}\definecolor{oucrimsonred}{rgb}{0.6, 0.0, 0.0}\definecolor{radicalred}{rgb}{1.0, 0.21, 0.37}\definecolor{raspberry}{rgb}{0.89, 0.04, 0.36}\definecolor{royalazure}{rgb}{0.0, 0.22, 0.66}
\definecolor{dex}{RGB}{138 18 34}
\definecolor{cs}{rgb}{0.0, 0.44, 1.0}

\theoremstyle{plain}
\newtheorem{theorem}{Theorem}[section]
\newtheorem{proposition}[theorem]{Proposition}
\newtheorem{lemma}[theorem]{Lemma}
\newtheorem{corollary}[theorem]{Corollary}

\theoremstyle{definition}

\theoremstyle{assumption}

\theoremstyle{remark}

\SetLabelAlign{center}{\hss#1\hss}

\def\blam{\boldsymbol{\lambda}}
\def\bbeta{\boldsymbol{\beta}}

\def\tr{\operatorname{tr}}

\def\C{\mathbf{C}}
\def\lass{\operatorname{lasso}}\def\slo{\operatorname{slope}}
\def\A{\boldsymbol{A}}
\def\bSigma{\mathbf{\Sigma}}
\def\B{\mathbf{B}}

\def\E{\mathbf{E}}

\def\R{\mathbb{R}}

\definecolor{RTGblue}{HTML}{003466}
\def\Z{\mathbf{Z}}

\def\X{\mathbf{X}}
\def\Z{\mathbf{Z}}

\def\P{\mathbb{P}}

\renewcommand{\L}{\operatorname{L}}

\newcommand{\ep}{\varepsilon}

\def\vec{\operatorname{vec}}
\newcommand{\e}{\mathrm{e}}

\renewcommand{\tilde}{\widetilde}%
\renewcommand{\d}{\mathrm{d}}
\newcommand{\qv}[1]{\langle#1\rangle}
\setkomafont{sectioning}{\normalcolor\bfseries}
\setkomafont{descriptionlabel}{\normalcolor\bfseries}
\setkomafont{author}{\large}
\setkomafont{date}{\normalsize}

\newlist{todolist}{itemize}{2}
\setlist[todolist]{label=$\square$}

\let\originalleft\left
\let\originalright\right
\renewcommand{\left}{\mathopen{}\mathclose\bgroup\originalleft}
\renewcommand{\right}{\aftergroup\egroup\originalright}

\definecolor{myteal}{RGB}{0 123 137}
\definecolor{radicalred}{rgb}{1.0, 0.21, 0.37}
\allowdisplaybreaks
	\bibliography{lasso}
\makeatother
\title{\fontsize{16}{19} \selectfont On Lasso and Slope drift estimators for Lévy-driven Ornstein--Uhlenbeck processes}
\author{Niklas Dexheimer\thanks{Aarhus University, Department of Mathematics, Ny Munkegade 118, 8000 Aarhus C, Denmark.
		\newline Email: \href{mailto:dexheimer@math.au.dk}{dexheimer@math.au.dk}} \qquad Claudia Strauch\thanks{Aarhus University, Department of Mathematics, Ny Munkegade 118, 8000 Aarhus C, Denmark.\newline Email: \href{mailto:strauch@math.au.dk}{strauch@math.au.dk}\newline CS gratefully acknowledges financial support of Sapere Aude: DFF-Starting Grant 0165-00061B
		“Learning diffusion dynamics and strategies for optimal control”.
	}} 
\begin{document}
	\maketitle
	\begin{abstract}
We investigate the problem of estimating the drift parameter of a high-dimensional Lévy-driven Ornstein--Uhlenbeck process under sparsity constraints.
It is shown that both Lasso and Slope estimators achieve the minimax optimal rate of convergence (up to numerical constants), for tuning parameters chosen independently of the confidence level, which improves the previously obtained results for standard Ornstein--Uhlenbeck processes. 
The results are nonasymptotic and hold both in probability and conditional expectation with respect to an event resembling the restricted eigenvalue condition.
	\end{abstract}

	\section{Introduction}\label{intro}
Due to increasing computational power, there has been an immense recent interest in high-dimensional statistical models, with many research efforts being made to understand statistical problems in a framework where the number of model parameters can be much larger than the number of observations.
For classical models such as linear regression, issues such as how to construct procedures which are both computationally efficient and show optimal statistical performance (as quantified in terms of convergence rates) are now well understood.
In contrast, only few deep statistical results are available as regards the high-dimensional modelling of continuous-time processes, even though these types of models can often be very well motivated from an application point of view.
A classic example of a continuous-time model of great practical relevance is the Ornstein--Uhlenbeck (OU) process, which is given as the solution of the stochastic differential equation (SDE)	
\begin{equation}\label{eq: SDE cont}
	\d X_t=-\A X_t\d t+\bSigma\d W_t, \quad t\geq0,
\end{equation}
where $\A,\bSigma\in\R^{d\times d}$ and $(W_t)_{t\geq0}$ is a $d$-dimensional Wiener process. 
In the scalar case, this process is referred to as the Vasicek model when applied to model interest rates.
In its multivariate version, it is frequently used, among many other applications, to model interbank lending (see \cite{fouque13}, \cite{carmona15}). 
Since the matrix $\A$ then describes the interactions between (a possibly very large number of) different banks, the question of estimating it from observations of $\X=(X_t)_{t\geq0}$ naturally arises. 
Given that banks often have only a limited number of lending partners, 
it is also natural to assume sparsity of $\A$, which is a classical assumption in the field of high-dimensional statistics as it allows to overcome the curse of dimensionality to some extent.
Given the availability of a continuous record of observations of \eqref{eq: SDE cont} with  $\bSigma = \operatorname{Id}_{d\times d}$ on some time interval $[0,T]$ and assuming sparsity of the interaction matrix, the issue of estimating $\A$ is investigated in \cite{gama19} and \cite{cmp20}.
The proposed estimators are of Lasso- (in the classical and its adaptive version) and Dantzig-type, since these estimators are known to induce sparse results. 

At first glance, it may come as a surprise that theoretical studies on high-dimensional versions of the basic model \eqref{eq: SDE cont} are relatively recent.
In fact, however, the investigation brings with it specific probabilistic challenges.
From the classical context of linear regression, it is well-known that convex penalisation methods (such as Lasso or Dantzig selectors) are efficient to compute, but show good statistical performance only under restrictive assumptions on the underlying design. 
An exemplary requirement for the Lasso estimator is the restricted eigenvalue property (see, e.g., (3.1), (4.2) and the beginning of Section 6 in \cite{belets18} or Section 3 in \cite{bickel09}).
Verifying a corresponding analogue in the context of continuous-time high-dimensional models amounts to the demanding task of establishing concentration of measure phenomena for unbounded functionals of the underlying process. 
In the Gaussian OU model, \cite{gama19} succeeded in showing by means of the $\log$-Sobolev inequality that the restricted eigenvalue property follows directly from the model assumptions as soon as $\A$ is symmetric, while \cite{cmp20} were even able to demonstrate (using Malliavin methods) that ergodicity of $\X$ already ensures the requested property.
Remarkably, unlike sparse linear regression, one thus does not have to impose the restricted eigenvalue property, as it can be derived directly from the model assumptions of the standard OU model. 
Based on this, high probability estimates for the Lasso estimator $\hat\A_{\lass}$ are proven in both \cite{gama19} and \cite{cmp20}.
In particular, denoting by $\Vert\cdot\Vert_2$ the Frobenius norm and by $s$ the sparsity of $\A$, Corollary 4.3 in \cite{cmp20} gives the tightest available bound by stating that there exists some constant $c>0$ such that
\begin{equation}\label{eq: result cmp}
	\Vert \hat{\A}_{\lass}-\A\Vert_2^2\leq \frac{c s}{T}\log\left(\frac{d^2}{\ep_0 } \right)
\end{equation}
holds true with probability larger than $1-\ep_0$, for observation time $T$ larger than some $T_0$ and adequately chosen tuning parameter, both depending on the confidence level $\ep_0 > 0$. 
The upper bound in \eqref{eq: result cmp} almost matches the well-known minimax optimal rate of estimation in sparse linear regression (see the introduction of \cite{belets18} and references therein), which in the given setting corresponds to
\begin{equation}\label{eq: minmax rate}
	\frac{s}{T}\log\left(\frac{d^2}{s}\right).
\end{equation}

The aims of this paper are now threefold. 
Firstly, we want to deduce the analysis of penalised estimators of the drift parameter for the more general class of Lévy-driven OU processes, i.e., we replace the driving Wiener process in \eqref{eq: SDE cont} by a general Lévy process $(Z_t)_{t\geq0}$, resulting in
\begin{equation}\label{eq: sde levy}
	\d X_t=-\A X_t\d t+\d Z_t, \quad t\geq0.
\end{equation}
Secondly, regarding the rates of convergence, we aim at closing the gap between \eqref{eq: result cmp} and \eqref{eq: minmax rate}, while also choosing the tuning parameter of the penalised estimators independently of the confidence level $\ep_0$, which corresponds to our third objective. 
To achieve the latter goals, a suitable candidate is the Slope estimator, introduced in \cite{bog15} as a weighted refinement of the Lasso estimator, which was shown to be minimax optimal for sparse linear regression in \cite{belets18}. 
Another result of this reference is that the Lasso estimator also attains the optimal convergence rate, but with the downside that the sparsity of the unknown parameter needs to be known for choosing suitable values for the tuning parameter, which is not the case for the Slope estimator. 
Furthermore, it is demonstrated that the tuning parameters for both Lasso and Slope estimators can be chosen independently of the confidence level $\ep_0$. 
At the heart of the proof of these results is a refined deviation inequality for the stochastic error term (Theorem 4.1 in \cite{belets18}), which in turn relies heavily on the (sub-) Gaussianity of the noise in the considered model. 
Since the stochastic error in the setting of \eqref{eq: sde levy} studied here corresponds to an It\={o} integral with non-deterministic integrand, reaching a result similar to the one obtained in the linear regression framework is not straightforward. 
For overcoming this challenge, we apply Talagrand's generic chaining device together with the restricted eigenvalue property, which then allows us to find a sufficiently tight result by bounding the Gaussian width of a given set.
In fact, using our methods, we succeed in defining estimators of the drift parameter $\A$ for the Lévy-driven OU process \eqref{eq: sde levy} that have the desired properties and, in particular, achieve the optimal convergence rate.

The structure of this paper is as follows. 
In Section \ref{sec: not}, we introduce the mathematical setting and notation of this paper, and in \ref{subsec: est int} we continue by introducing the two bespoke estimators. 
Section \ref{sec: results} contains our main results on the performance of both Lasso and Slope estimators in the form of oracle inequalities resp.\ bounds in various norms. 
We also discuss the optimality of the derived upper bounds on the rates of convergence.
The subsequent section consists of the required deviation inequalities for the results in Section \ref{sec: results}, namely a property of restricted eigenvalue-type (Section \ref{sec: ret prop}) and the aforementioned deviation inequality for the stochastic error term (Section \ref{sec: dev chaining}). 
As explained in Section \ref{sec: results}, our results rely on the concentration assumption \ref{ass: concentration}, which is discussed in more detail in Section \ref{sec: ass conc}. 
The paper concludes by a brief simulation study in Section \ref{sec: sim}, where we compare the error of the maximum likelihood estimator to both Lasso and Slope estimators in various dimensions.
The appendix contains basic probabilistic results for the processes considered, as well as some longer proofs.

\subsection{Preliminaries and notation} \label{sec: not}
In the following, $\Z=(Z_t)_{t\ge0}$ will denote a $d$-dimensional L\'evy process on a given filtered probability space $(\Omega,\mathscr F,(\mathscr F_t),\P)$, adapted to the filtration $(\mathscr F_t)_{t\ge0}$. 
For $\A\in\R^{d\times d}$, we call a strong solution $\X=(X_t)_{t\geq0}$ of the SDE
\begin{equation}\label{sde:ou}
	\d X_t=-\A X_t\d t+\d Z_t,\quad t>0,\ X_0\sim \pi,
\end{equation}
an Ornstein--Uhlenbeck (OU) process with background driving Lévy process (BDLP) $\Z$, initial distribution $\pi$ and parameter $\A$. 
The initial condition $X_0$ is assumed to be independent of $\Z$.
It follows from It\={o}'s formula that an explicit solution of \eqref{sde:ou} is given by (see e.g.\ equations (1.1) and (1.2) in \cite{mas04})
\begin{equation}\label{eq: ou explicit}
	X_t=\e^{-t\A}X_0+\int_0^t\e^{-(t-s)\A}\d Z_s,\quad t>0.
\end{equation}
Denote by $(b,\C=\bSigma\bSigma^\top,\nu)$ the generating triplet of $\Z$, i.e., $b=(b_k)_{k=1}^d\in\R^d$, $\C=\bSigma\bSigma^\top=(C_{kl})_{k,l=1}^d$ is a $d\times d$ symmetric non-negative definite matrix and $\nu$ is a Lévy measure, i.e., a $\sigma$-finite measure on $\R^d$ satisfying 
\[\nu(\{\boldsymbol 0\})=0\quad\text{ and }\quad \int_{\R^d}\min\{1,\|z\|^2\}\nu(\d z)<\infty.\]
Recall that, by the Lévy--It\={o} decomposition (see e.g.\ Theorem 2.4.16 in \cite{applebaum09}), it then holds 
\[
Z_t=bt+\Sigma W_t+\int_0^t\int_{\Vert z\Vert\geq1} zN(\d s,\d z)+\int_0^t\int_{\Vert z\Vert<1} z\tilde{N}(\d s,\d z), \quad t\geq 0,
\]
where $(W_s)_{s\geq 0}$ denotes a $d$-dimensional Wiener process, $N$ is a Poisson random measure on $[0,\infty)\times \R^d$ with intensity measure given by $\mathbf{\lambda}\otimes \nu$, and $\tilde{N}$ denotes its compensated counterpart.
Furthermore, denote by $\P^{\A}_t$ the restriction of the measure $\P^{\A}$ induced by \eqref{sde:ou} on the path space to $\mathscr F_t$.   
For $\A\in\R^{d_1\times d_2}$, we define
\[
\|\A\|_0\coloneq \sum_{1\le i\le d_1,1\le j\le d_2}\mathds{1}\{A_{ij}\ne0\}, \qquad 
\|\A\|_p\coloneq \left(\sum_{1\le i\le d_1,1\le j\le d_2}|A_{ij}|^p\right)^{1/p}, \quad p\ge 1,
\]
and set $\Vert \A\Vert_{\mathrm{Sp}}$ to be the spectral norm.
To the Frobenius norm $\|\cdot\|_2$, we associate the scalar product
\[
\langle \A_1,\A_2\rangle_{2}\coloneq \tr\left(\A_1^\top\A_2\right),\quad \A_1,\A_2\in\R^{d_1\times d_2},
\]
for $\tr(\cdot)$ denoting the trace. 
Additionally, for $p\in[1,\infty)\cup\{0\}$ and $r>0$, set 
\[
\mathbb{B}_p(r)\coloneq \left\{\B\in\R^{d\times d}: \Vert \B\Vert_{p}\leq r \right\}.
\]
For a symmetric matrix $\A\in\R^{d\times d}$, we write $\lambda_{\max}(\A)$ and $\lambda_{\min}(\A)$ for the largest and the smallest eigenvalue of $\A$, respectively. 
Denote by $M_+(\R^d)$ the set of all real $d\times d$ matrices such that the real parts of all eigenvalues are positive, i.e., $\A\in M_+(\R^d)$ if and only if $\|\e^{-t\A}\|_2\to0$ as $t\to\infty$. 

Given $\bbeta=(\beta_1,\ldots,\beta_d)\in\R^d$, denote by $(\beta_1^\#,\ldots,\beta_d^\#)$ a nonincreasing rearrangement of $|\beta_1|,\ldots,|\beta_d|$.
For a vector of tuning parameters $\blam=(\lambda_1,\ldots,\lambda_d)\in\R^d$ not all equal to 0 such that $\lambda_1\ge\lambda_2\ge\ldots\ge \lambda_d\ge0$, we set
\[
\|\bbeta\|_\ast\coloneq \sum_{j=1}^d\lambda_j\beta_j^\#,\quad \bbeta\in\R^d.\]
Then it is known that $\Vert \cdot\Vert_*$ defines a norm on $\R^d$ (see Proposition 1.2 in \cite{bog15}). 
In the following, the weights will always be given by 
\[\lambda_j= \sqrt{\log\left(\frac{2d}{j}\right)},\quad j\in\{1,\ldots d\}.\]
For $\A\in \R^{d_1\times d_2}$, we set (by a slight abuse of notation) $\Vert \A\Vert_*\coloneq \Vert \vec(\A)\Vert_*,$ i.e.,
\begin{equation}\label{def:normstar}
	\Vert \A\Vert_*=\sum_{j=1}^{d_1d_2}\vec(\A)^\#_j\sqrt{\log\left(\frac{2d_1d_2}{j}\right)}. 
\end{equation}
Finally, for stochastic processes $(X_t)_{t\in[0,T]},(Y_t)_{t\in[0,T]}\in L^2([0,T],\d t)$, we introduce the scalar product
\[
\langle X,Y\rangle_{L^2}\coloneq \frac1T\int_0^TX_s^\top Y_s\d s.
\]

\subsection{The Lasso and Slope estimators} \label{subsec: est int}
In the following, we assume that a continuous record of observations up to time $T>0$ of a Lévy-driven OU process $\X$ is available, and the goal is to estimate the unknown true drift parameter $\A_0\in\R^{d\times d}$. 
Additionally, we assume that the corresponding path of the continuous martingale part $\X^\mathsf{c}=(X^\mathsf{c}_t)_{t\geq0}$ of $\X$ and the diffusion parameter $\bSigma$ are known. 
Extraction of the continous martingale part from discrete observations of $\X$ by employing a truncation approach was discussed in \cite{mai14} in the context of maximum likelihood estimation. 

To begin with our analysis, we introduce the following assumption, which will be in place \textbf{implicitly throughout the whole paper}.
\begin{enumerate}[$(\mathcal{A}_0$)]
	\item \label{ass: ergodicity}
	 $\A_0\in M_+(\R^d)$, $\C$ is strictly positive definite, and $\nu$ admits a second moment. 
	Additionally, it holds $\pi=\mu,$ i.e., $\X$ is stationary.
\end{enumerate}

Of course for \ref{ass: ergodicity} to make sense, an invariant distribution has to exist. 
It is, however, well known that this is the case if $\A_0\in M_+(\R^d)$ and $\E[(1\vee\log(\Vert Z_1\Vert))]$ is finite (see Theorems 4.1 and 4.2 in \cite{sato84} or Proposition 2.2 in \cite{mas04}).

Under \ref{ass: ergodicity}, we are able to employ the results of \cite{sorensen91} where maximum likelihood estimation for general jump diffusions is investigated. 
As condition \textbf{C} of \cite{sorensen91} clearly follows from \ref{ass: ergodicity}, we get the following result.

\begin{proposition}\label{prop: sorensen likelihood}
	Let $\A\in\R^{d\times d}, T\geq0$. 
	Then, 
	\begin{equation}\label{def:like0}
		\frac{\d \P_T^{\A}}{\d \P_T^{\boldsymbol{0}}}=\exp\left(-\int_0^T  ( \C^{-1} \A X_{s-})^\top  \d X^{\mathsf{c}}_s-\frac 1 2 \int_0^T(\bSigma^{-1} \A X_{s-})^\top  \bSigma^{-1} \A X_{s-} \d s \right). 
	\end{equation}
\end{proposition}
\begin{proof}
	From \eqref{eq: ou explicit}, we have by the Lévy--It\={o} decomposition that, under $\P^{\A}$,
	\[
	X_t=\e^{-tA}X_0+\int_0^t \e^{-(t-s)\A}b^*\d s+\int_0^t \e^{-(t-s)\A} \bSigma \d W_s +\int_0^t\int_{\R^d} \e^{-(t-s)\A} z\tilde N(\d s, \d z),
	\]
	where everything is given as in Section \ref{sec: not} and
	$$b^*\coloneq b+\int_{\Vert z\Vert >1}z\nu(\d z).$$ 
	Thus, Corollary 4.4.24 in \cite{applebaum09} implies that $\E[\sup_{0\leq s\leq t}\Vert X_s\Vert^2]$ is finite, since $\mu$ and $\nu$ admit a second moment by \ref{ass: ergodicity} and Corollary \ref{cor: inv dens moments}. 
	Hence, we get 
	\[\P^{\A}\left(\int_0^T (\bSigma^{-1} \A X_{s-})^\top \bSigma^{-1} \A X_{s-}\d s<\infty\right)= \P^{\boldsymbol{0}}\left(\int_0^T (\bSigma^{-1} \A X_{s-})^\top \bSigma^{-1} \A X_{s-} \d s<\infty\right)=1,\]
	where we argued analogously for $\P^{\boldsymbol{0}}$. 
	This concludes the proof by Theorem 2.1 in \cite{sorensen91}.
\end{proof}

Given \eqref{def:like0}, we are able to determine the likelihood function and thus can define the Lasso and Slope estimators. 
For doing so, we set
\begin{equation}\label{def:like}
\mathcal L_T(\A)\coloneq -\frac1T\log\left(\frac{\d\P_T^{\A}}{\d\P_T^{\boldsymbol 0}}\right), \quad \A\in\R^{d\times d}.
\end{equation}
Furthermore, as we do \emph{not} assume $\bSigma$ to be the identity (as in \cite{gama19} and \cite{cmp20}) or the identity matrix multiplied by some factor (as in \cite{belets18}), we have to adjust the classical definitions of Lasso and Slope estimators slightly in our setting.
We define the Lasso estimator to be given as
\begin{equation}\label{def:lasso}
	\hat{\A}_{\lass}\in \operatorname{argmin}_{\A\in\R^{d\times d}}\left(\mathcal L_T(\A)+\lambda_{\L}\|\bSigma^{-1}\A\|_1\right),
\end{equation}
where $\lambda_{\L}>0$ is a tuning parameter.
For the Slope estimator, we set 
\begin{equation}\label{def: slope}
	\hat{\A}_{\slo}\in \operatorname{argmin}_{\A\in\R^{d\times d}}\left(\mathcal L_T(\A)+\lambda_{\mathrm{S}}\|\bSigma^{-1}\A\|_\ast\right),
\end{equation}
where again $\lambda_{\mathrm{S}}>0$ is a tuning parameter and $\Vert \cdot\Vert_*$ is defined in \eqref{def:normstar}. 
Our interest in this estimator is motivated by the fact that, in the classical context of high-dimensional linear regression on the class of $s$-sparse vectors in $\R^d,$ the Slope estimator with suitably chosen tuning parameters achieves the optimal rate $(s/n)\log(d/s)$, $n$ denoting the number of observations, for both the prediction and the $\ell_2$ estimation risks under suitable assumptions. As both estimators are defined as a solution of a convex optimization problem, they can be computed efficiently.	

\section{Probability estimates for the Lasso and Slope estimators} \label{sec: results}
The goal of this section is to provide probability estimates for the performance of Lasso and Slope estimators with tuning parameters not tied to a confidence level.
The starting point of both proofs is given by the following auxiliary result. 

\begin{lemma}\label{lemma: L2 frobenius}
	Let $h\colon \R^{d\times d}\to\R$ be a convex function, and recall the definition of $\mathcal L_T(\cdot)$ in \eqref{def:like}.
	If $\hat{\A}$ is a solution of the minimization problem $\min_{\A\in\R^{d\times d}}\left(\mathcal L_T(\A)+h(\A)\right)$, then $\hat{\A}$ satisfies for all $\A\in\R^{d\times d}$ 
	\[
	\|\bSigma^{-1}(\hat{\A}-\A_0)X\|_{L^2}^2-\|\bSigma^{-1}(\A-\A_0)X\|_{L^2}^2\le 2(\langle\boldsymbol{\ep}_T,\bSigma^{-1}(\A-\hat{\A})\rangle_2+h(\A)-h(\hat{\A}))-\|\bSigma^{-1}(\hat{\A}-\A)X\|_{L^2}^2,
	\]
	where 
	\begin{equation}\label{def:ep}
		\boldsymbol{\ep}^\top_T\coloneq\frac 1 T \int_0^T X_s\d \tilde{W	}_s^\top, 
	\end{equation}
	with $(\tilde{W}_s)_{s\geq 0}$ being a $\P^{\A_0}$-Wiener process.
\end{lemma}
The proof of Lemma \ref{lemma: L2 frobenius} relies on the convexity of $\mathcal{L}_T$ and $h$, combined with an application of Girsanov's theorem, and can be found in Appendix \ref{app: res proof}. 
Additionally, the proofs for our main results on the performance of the estimators require deviation inequalities for $\boldsymbol{\ep}_T$ and properties resembling the restricted eigenvalue property, which is a classical assumption in the context of linear regression. 
These results can be found in Section \ref{sec: dev ineq}.
They are based on the following assumption:

\begin{enumerate}[$(\mathcal{H})$]
	\item\label{ass: concentration}
	There exists a function $H\colon \R^+\times \R^+\to \R^+$ such that
	\begin{enumerate}[label=(\roman*)]
		\item for any $T,r>0$, the functions $H(T,\cdot)$ and $H(\cdot,r)$ are non-increasing and such that $\lim\limits_{T\to\infty} H(T,r)=0$ for all $r>0$, and
		\item for any vector $u\in \R^d$ with $\Vert u\Vert\leq 1$, it holds
		\[\forall T,r>0, \quad \P\left(\vert u^\top (\hat{\C}_T-\C_\infty)u\vert \geq r \right)\leq H(T,r),\]
		where 
		\begin{equation}\label{def:C}
			\hat{\C}_T\coloneq \frac 1 T \int_0^T X_s X_s^\top \d s\quad \text{ and }\quad  \C_\infty\coloneq \int xx^\top \mu(\d x).
		\end{equation}
	\end{enumerate}
\end{enumerate}
Let $\kappa_{\min}\coloneqq \lambda_{\min}(\C_\infty)$ and  $\kappa_{\max}\coloneqq \lambda_{\max}(\C_\infty)$.
Note that $0<\kappa_{\min}\leq \kappa_{\max}$ holds because of \ref{ass: ergodicity} (see the remark at the end of Appendix \ref{app: levy facts}). 
For easing the notation, we also introduce the events
\begin{equation}\label{def:Q}
	Q_T(r)\coloneq \left\{\sup_{\B\in\mathbb{B}_2(1) }\vert \tr(\B (\hat{\C}_T-\C_\infty)\B^\top)\vert\leq r \right\},\quad r>0. 
\end{equation}
In Proposition \ref{prop: REP}, we will see that \ref{ass: concentration} in fact implies a lower bound on $\P(Q_T(r))$ for any $r>0$.

It was shown in \cite{gama19} and \cite{cmp20} in the Gaussian OU case that the restricted eigenvalue property holds with high probability for large enough values of $T$ and thus follows implicitly from the model as soon as assumption \ref{ass: concentration} is in place. 
In Section \ref{sec: ass conc}, we investigate \ref{ass: concentration} in more detail by providing sufficient conditions in the Lévy-driven case for \ref{ass: concentration} to hold and recalling the results in the Gaussian case.

\subsection{Main results on the Lasso estimator}
A notable feature of many nonasymptotic bounds for the Lasso estimator available in the literature (see Corollary 1 in \cite{gama19} or Corollary 4.3 in \cite{cmp20}) is that the confidence level is tied to the tuning parameter $\lambda$.
In the high-dimensional linear regression model, \cite{belets18} develop new proof strategies for the Lasso estimator, which in particular allow to derive bounds in probability at \emph{any} level of confidence with the same tuning parameter.
We now adapt their findings to the high-dimensional Lévy-driven OU model considered in this paper and show that here, too, there is no need for the confidence level to be linked to the tuning parameter.

\begin{proposition}\label{prop: Lasso}
	Grant Assumption \ref{ass: concentration}.
		Set $s=\Vert \bSigma^{-1}\A_0\Vert_0$, and let $\hat{\A}=\hat{\A}_{\lass}$ be the Lasso estimator \eqref{def:lasso} with tuning parameter
		\begin{equation}\label{low:lambda}
			\lambda_T\geq 2c_* \sqrt{\frac{\kappa_{\max}}{T}\log\left( \frac{2\e d^2}{s}\right)},
		\end{equation}
		where $c_*$ is defined in Proposition \ref{prop: chain slope}.
		Then, for any $\A\in\R^{d\times d}$ satisfying $\Vert \bSigma^{-1}\A\Vert_0\le s$, the upper bound
		\[
		\|\bSigma^{-1}(\hat{\A}-\A_0)X\|_{L^2}^2+2\lambda_T\Vert \bSigma^{-1}(\A-\hat{\A})\Vert_1
		\le\|\bSigma^{-1}(\A-\A_0)X\|_{L^2}^2+\frac{8s\lambda_T^2}{\kappa_{\min}}\left(1\vee\frac{\log(4\ep_0^{-1})}{s \log(2\e d^2/s)}  \right)
		\]
		holds with probability of at least 
		\[
		1-\frac{\ep_0}{2}-(21(d\land\e))^d H\left(T,\frac{\kappa_{\min}}{6d}\right),
		\] 
		for all $\ep_0\in(0,1)$ and $T>0$.
	\end{proposition}
	\begin{proof}
		By Propositions \ref{prop: REP} and \ref{prop: chain slope}, it holds
		\begin{equation}\begin{split}\label{eq: lasso cond}
			&1-\frac{\ep_0}{2}-(21(d\land\e))^d H\left(T,\frac{\kappa_{\min}}{6d}\right)
			\\&\hspace*{3em}\le\P\left(\inf_{\B\in\R^{d\times d}\backslash\{0\} }\frac{\Vert \B X\Vert_{L^2}^2}{\Vert \B\Vert_{2}^2}\geq \frac{\kappa_{\min}}{2},\sup_{\B\in \R^{d\times d}, \B\neq 0}\frac{\qv{\boldsymbol{\ep}_T,\B}_{2}}{\Vert \B\Vert_S}\leq c_* \sqrt{\frac{\kappa_{\max}}{T}} \right),
		\end{split}\end{equation}
		where $\Vert \cdot\Vert_S$ is defined in \eqref{def:norms}.
		From now on, we assume that the event on the rhs of \eqref{eq: lasso cond} occurs, and we fix $\A\in\R^{d\times d}$.
		By Lemma \ref{lemma: L2 frobenius}, we then have for $h(\cdot)= \lambda_T \Vert\bSigma^{-1} \cdot \Vert_1$ and $\B\coloneq \A-\hat{\A}$,
		\begin{align}\label{eq: lasso opt proof start}
			\|\bSigma^{-1}(\hat{\A}-\A_0)X\|_{L^2}^2+\lambda_T\Vert \bSigma^{-1}\B\Vert_1
			\le\|\bSigma^{-1}(\A-\A_0)X\|_{L^2}^2-\|\bSigma^{-1}\B X\|_{L^2}^2+\Delta_*,
		\end{align}
		where 
		\[\Delta_*\coloneq\lambda_T\Vert \bSigma^{-1}\B\Vert_1+2\langle\boldsymbol{\ep}_T,\bSigma^{-1}\B\rangle_{2}+2\lambda_T\left(\Vert\bSigma^{-1}\A\Vert_1-\Vert\bSigma^{-1}\hat{\A}\Vert_1\right).\] 
		Now note that, due to the Cauchy--Schwarz inequality and equation (2.7) in \cite{belets18}, for any $s\in\left\{1,\ldots,d^2\right\}$,
		\begin{align}
			\Vert\bSigma^{-1}\B\Vert_*&\leq\Vert \bSigma^{-1}\B\Vert_{2}\sqrt{\sum_{i=1}^s \log(2d^2/i) }+\sum_{i=s+1}^{d^2}\vec(\bSigma^{-1}\B)^\#_i \sqrt{\log(2d^2/i})\notag
			\\&\leq \sqrt{\log(2\e d^2/s)}\left(\sqrt{s}\Vert\label{eq: G F} \bSigma^{-1}\B\Vert_{2}+\sum_{i=s+1}^{d^2}\vec(\bSigma^{-1}\B)^\#_i  \right)\eqqcolon F(\bSigma^{-1}\B).
		\end{align}
		Hence, by Lemma A.1 in \cite{belets18}, if $\Vert \bSigma^{-1}\A\Vert_0\leq s$, the term $\Delta_\ast$ is bounded from above by
		\begin{align*}
	&2\lambda_T\left(\frac{1}{2}\Vert \bSigma^{-1}\B\Vert_1+\Vert\bSigma^{-1}\A\Vert_1-\Vert\bSigma^{-1}\hat{\A}\Vert_1\right)\\
			&\hspace*{3em}+\left(\log\left(\frac{2\e d^2}{s}\right)\right)^{-\frac12}\lambda_T\left(F(\bSigma^{-1} \B)\vee\sqrt{\log(4\ep_0^{-1})} \Vert \bSigma^{-1}\B\Vert_{2}\right)
			\\
			& \leq \lambda_T\left(3\sqrt{s}\Vert \bSigma^{-1}\B\Vert_{2}-\sum_{i=s+1}^{d^2}\vec(\B)^\#_i\right)+\left(\log\left(\frac{2\e d^2}{s}\right)\right)^{-\frac12}\lambda_T\left(F(\bSigma^{-1} \B)\vee\sqrt{\frac{2\log(4\ep_0^{-1})}{\kappa_{\min}}} \Vert \bSigma^{-1}\B X\Vert_{L^2}\right),
		\end{align*}
		where we also used \eqref{eq: lasso cond} and $2c_*\sqrt{T^{-1}\kappa_{\max}}\leq(\log(2\e d^2/s))^{-1/2}\lambda_T$. 
		We continue by examining the two different cases which can occur due to the maximum term. 
		On the one hand,
		 \[\sqrt{\frac{2\log(4\ep_0^{-1})}{\kappa_{\min}}} \Vert \bSigma^{-1}\B X\Vert_{L^2}\geq F(\bSigma^{-1}\B) \quad\implies \quad
		\Vert \bSigma^{-1}\B\Vert_{2}\leq \sqrt{\frac{2\log(4\ep_0^{-1})}{s\log(2\e d^2/s)\kappa_{\min} }}\Vert \bSigma^{-1}\B X\Vert_{L^2}.\]
		Thus, in this case,
		\begin{align*}
			\Delta_*&\leq3\lambda_T\sqrt{s}\Vert \bSigma^{-1}\B\Vert_{2}+\lambda_T\sqrt{\frac{2\log(4\ep_0^{-1})}{\log(2\e d^2/s)\kappa_{\min}}} \Vert \bSigma^{-1}\B X\Vert_{L^2}
			\\
			&\leq 4\lambda_T\sqrt{\frac{2\log(4\ep_0^{-1})}{ \log(2\e d^2/s)\kappa_{\min}}}\Vert \bSigma^{-1}\B X\Vert_{L^2} 
			\\
			&\leq 8\lambda_T^2\frac{\log(4\ep_0^{-1})}{ \log(2\e d^2/s)\kappa_{\min}}+\Vert \bSigma^{-1}\B X\Vert^2_{L^2}.
		\end{align*}
		In the opposite case,
		\begin{align*}
			\Delta_*&\leq  4\lambda_T\sqrt{s}\Vert \bSigma^{-1}\B\Vert_{2}
			\\
			&\leq 4\lambda_T\sqrt{\frac{2s}{\kappa_{\min}}}\Vert \bSigma^{-1}\B X\Vert_{L^2}
			\\
			&\leq\frac{8s \lambda_T^2}{\kappa_{\min}}+\Vert \bSigma^{-1}\B X\Vert^2_{L^2}.
		\end{align*}
		Inserting these results into \eqref{eq: lasso opt proof start} yields
		\[
		\|\bSigma^{-1}(\hat{\A}-\A_0)X\|_{L^2}^2+2\lambda_T\Vert \bSigma^{-1}\B\Vert_1
		\le\|\bSigma^{-1}(\A-\A_0)X\|_{L^2}^2+\frac{8\lambda_T^2}{\kappa_{\min}}\left(s\vee\frac{\log(4\ep_0^{-1})}{ \log(2\e d^2/s)}  \right).
		\]
	\end{proof}

While the lower bound \eqref{low:lambda} for the specification of the tuning parameter $\lambda_T$ does not depend on $\ep_0$, as promised, the sparsity $s$ of $\bSigma^{-1}\A_0$ appears there, which is of course unknown in general.
As, however, $\bSigma$ and $\A_0$ are invertible by assumption, it always holds $s\geq d$.
Thus, choosing
	\[
	\lambda_T\geq 2c_* \sqrt{T^{-1}\kappa_{\max}\log\left( 2\e d\right)}
\]
implies the conditions in Proposition \ref{prop: Lasso} to be fulfilled. 
This specification leads to the same rate for the Lasso estimator as derived in \cite{gama19} and \cite{cmp20} for Gaussian OU processes. 
We also find this in the following result, where we apply Proposition \ref{prop: Lasso} for getting high probability estimates in various norms.

	\begin{corollary}\label{cor: lasso}
		Let everything be given as in Proposition \ref{prop: Lasso}, and let $\ep_1\in(0,1)$. 
		Then, for 
		\begin{equation}\label{def:T0}
			T>T_0(\ep_1)\coloneq\inf\left\{T>0: (21(d\land\e))^d H\left(T,\frac{\kappa_{\min}}{6d}\right)\leq\frac{\ep_1}{2}\right\},
		\end{equation}
		the following assertions hold, each with probability larger than $1-\tfrac{1}{2}(\ep_0+\ep_1)$, for all $\ep_0\in(0,1)$:
\leqnomode
\begin{align}\label{lasso:a}\tag{$\operatorname{a}$}
		\|\bSigma^{-1}(\hat{\A}-\A_0)X\|_{L^2}^2
		&\le\frac{8s\lambda_T^2}{\kappa_{\min}}\left(1\vee\frac{\log(4\ep_0^{-1})}{ s\log(2\e d^2/s)}  \right),\\\label{lasso:b}\tag{$\operatorname{b}$}
		\|\bSigma^{-1}(\hat{\A}-\A_0)\|_{2}^2
		&\le\frac{16s\lambda_T^2}{\kappa_{\min}^2}\left(1\vee\frac{\log(4\ep_0^{-1})}{ s\log(2\e d^2/s)}  \right),\\\label{lasso:c}\tag{$\operatorname{c}$}
		\Vert \bSigma^{-1}(\hat{\A}-\A_0)\Vert_1
		&\le\frac{4s\lambda_T}{\kappa_{\min}}\left(1\vee\frac{\log(4\ep_0^{-1})}{ s\log(2\e d^2/s)}  \right).
\end{align}	
	\end{corollary}
	\begin{proof}
		Assertions \eqref{lasso:a} and \eqref{lasso:b} follow immediately by applying Proposition \ref{prop: Lasso} with $\A=\A_0$. 
		For \eqref{lasso:b}, note that \eqref{eq: lasso cond} implies $\kappa_{\min}\Vert \bSigma^{-1}(\hat{\A}-\A_0)\Vert_{2}^2\leq2 \|\bSigma^{-1}(\hat{\A}-\A_0)X\|^2_{L^2},$ and hence the assertion follows by \eqref{lasso:a}.
	\end{proof}

\subsection{Main results on the Slope estimator}
We now state our main result on the Slope estimator in an analogous form to Proposition \ref{prop: Lasso}. 

	\begin{proposition}\label{prop: slope}
		Set $s=\Vert \bSigma^{-1}\A_0\Vert_0$, and choose $c_S\geq  c_*\sqrt{\kappa_{\max}}$, where $c_*$ is defined in Proposition \ref{prop: chain slope}.
		Let $\hat{\A}=\hat{\A}_{\slo}$ be the Slope estimator \eqref{def: slope} with tuning parameter
		\begin{equation}\label{def:lambdaT}
   \lambda_T\coloneq \frac{2c_S} {\sqrt{T}}.
		\end{equation}
		Then, for any $\A\in\R^{d\times d}$ satisfying $\Vert \bSigma^{-1}\A\Vert_0\leq s$, the inequality
		\[
		\|\bSigma^{-1}(\hat{\A}-\A_0)X\|_{L^2}^2+\frac{2c_S \Vert \bSigma^{-1}(\hat{\A}-\A)\Vert_*}{\sqrt{T}}
		\le\|\bSigma^{-1}(\A-\A_0)X\|_{L^2}^2+32c^2_S \frac{s\log\left(\frac{2\e d^2}{s} \right)}{T\kappa_{\min}}\left(1\vee \frac{\log(4\ep_0^{-1})}{s\log\left(\frac{2\e d^2}{s}\right)}\right)
		\]
		holds with probability of at least $1-\ep_0/2-(21(d\land\e))^d H\left(T,\frac{\kappa_{\min}}{6d}\right)$, for all $\ep_0\in(0,1)$ and $T>0$.
	\end{proposition}
Note that (by the choice of $\lambda_T$ in \eqref{def:lambdaT}) the Slope estimator achieves the stated (optimal) rate of convergence even if the sparsity $s$ of $\bSigma^{-1}\A_0$ is not known.
	\begin{proof}[Proof of Proposition \ref{prop: slope}]
		As in the proof of Proposition \ref{prop: Lasso},
we start with inequality \eqref{eq: lasso cond}, and we assume that the event appearing in the upper bound holds true.
By Lemma \ref{lemma: L2 frobenius}, we then get for $h(\cdot)=2c_S \sqrt{T^{-1}}\Vert \bSigma^{-1}\cdot\Vert_*$ and all $\A\in\R^{d\times d}$,
		\[\|\bSigma^{-1}(\hat{\A}-\A_0)X\|_{L^2}^2-\|\bSigma^{-1}(\A-\A_0)X\|_{L^2}^2\le 2(\langle\boldsymbol{\ep}_T,\bSigma^{-1}(\A-\hat{\A})\rangle_F+h(\A)-h(\hat{\A}))-\|\bSigma^{-1}(\hat{\A}-\A)X\|_{L^2}^2.\]
		Thus, for $\B\coloneqq\A-\hat{\A}$,
		\begin{align}\label{eq: slope start}
			2c_S \sqrt{T^{-1}}\Vert \bSigma^{-1}\B\Vert_*+ \|\bSigma^{-1}(\hat{\A}-\A_0)X\|_{L^2}^2
			\le\|\bSigma^{-1}(\A-\A_0)X\|_{L^2}^2-\Vert\bSigma^{-1}\B X\|_{L^2}^2+\Delta^*,
		\end{align}
		where 
		\[
		\Delta^\ast
		\coloneqq
		2\langle\boldsymbol{\ep}_T,\bSigma^{-1}\B\rangle_F+4c_S \sqrt{T^{-1}}\left(\Vert\bSigma^{-1}\A\Vert_*-\Vert\bSigma^{-1}\hat{\A}\Vert_* +\frac12\Vert\bSigma^{-1} \B\Vert_*\right).
		\]
		Now, \eqref{eq: lasso cond} and Lemma A.1 in \cite{belets18} imply, if $\Vert \bSigma^{-1}\A\Vert_0\leq s$, that
		\begin{align*}
			\Delta^\ast&\leq 4c_S \sqrt{T^{-1}}\left(\frac 1 2 \Vert \bSigma^{-1}\B\Vert_S+\Vert\bSigma^{-1}\A\Vert_*-\Vert\bSigma^{-1}\hat{\A}\Vert_* +\frac{1}{2} \Vert \bSigma^{-1}\B\Vert_*\right)
			\\
			&\leq 4c_S \sqrt{T^{-1}}\left(\frac 1 2\Vert \bSigma^{-1}\B\Vert_S+\frac{3}{2}\sqrt{\sum_{i=1}^s\log\left(\frac{2d^2}{i} \right) }\Vert \bSigma^{-1}\B\Vert_{2}-\frac{1}{2}\sum_{i=s+1}^{d^2}\vec(\bSigma^{-1}\B)^\#_i\sqrt{\log\left(\frac{2d^2}{i} \right)}\right)
			\\
			&\leq 4c_S \sqrt{T^{-1}}\left(\frac 1 2\Vert \bSigma^{-1}\B\Vert_S+\frac{3}{2}\sqrt{s\log\left(\frac{2\e d^2}{s} \right) }\Vert \bSigma^{-1}\B\Vert_{2}-\frac{1}{2}\sum_{i=s+1}^{d^2}\vec(\bSigma^{-1}\B)^\#_i\sqrt{\log\left(\frac{2d^2}{i} \right)}\right),	
		\end{align*}
		where we used equation (2.4) in \cite{belets18} in the last step. 
		Furthermore, arguing as in the derivation of equation \eqref{eq: G F} in the proof of Proposition \ref{prop: Lasso} and using \eqref{eq: lasso cond}, we arrive at
		\begin{align*}
			\Delta^\ast
			\leq 4c_S \sqrt{T^{-1}}\Bigg(&\frac 1 2\left(F(\bSigma^{-1} \B)\vee\sqrt{\frac{2\log(4\ep_0^{-1})}{\kappa_{\min}}} \Vert \bSigma^{-1}\B X\Vert_{L^2}\right)\\&+\frac{3}{2}\sqrt{s\log\left(\frac{2\e d^2}{s} \right) }\Vert \bSigma^{-1}\B\Vert_{2}-\frac{1}{2}\sum_{i=s+1}^{d^2}\vec(\bSigma^{-1}\B)^\#_i\sqrt{\log\left(\frac{2d^2}{i} \right)}\Bigg).
		\end{align*}
		Again, we continue by investigating the two different cases related to the maximum term.
		Firstly, 
		\[\sqrt{\frac{2\log(4\ep_0^{-1})}{\kappa_{\min}}} \Vert \bSigma^{-1}\B X\Vert_{L^2}\geq F(\bSigma^{-1}\B) \quad \implies \quad
		\Vert \bSigma^{-1}\B\Vert_{2}\leq \sqrt{\frac{2\log(4\ep_0^{-1})}{s\log(2\e d^2/s)\kappa_{\min} }}\Vert \bSigma^{-1}\B X\Vert_{L^2},\]
		and hence 
		\begin{align*}
			\Delta^\ast
			&\leq 4c_S \sqrt{T^{-1}}\Bigg(\frac 1 2\sqrt{\frac{2\log(4\ep_0^{-1})}{\kappa_{\min}}} \Vert \bSigma^{-1}\B X\Vert_{L^2}+\frac{3}{2}\sqrt{s\log\left(\frac{2\e d^2}{s} \right) }\Vert \B\Vert_{2}\Bigg)
			\\
			&\leq  8c_S\sqrt{\frac{2\log(4\ep_0^{-1})}{T\kappa_{\min}}} \Vert \bSigma^{-1}\B X\Vert_{L^2}
			\\
			&\leq 32 c_S^2\frac{\log(4\ep_0^{-1})}{T\kappa_{\min}}+ \Vert \bSigma^{-1}\B X\Vert_{L^2}^2.
		\end{align*}
		In the other case, we get
		\begin{align*}
			\Delta^*
			&\leq 4c_S \sqrt{T^{-1}}\Bigg(\frac 1 2 F(\bSigma^{-1} \B)+\frac{3}{2}\sqrt{s\log\left(\frac{2\e d^2}{s} \right) }\Vert \bSigma^{-1}\B\Vert_{2}-\frac{1}{2}\sum_{i=s+1}^{d^2}\vec(\bSigma^{-1}\B)^\#_i\sqrt{\log\left(\frac{2d^2}{i} \right)}\Bigg)
			\\
			&\leq 8c_S \sqrt{T^{-1}s\log\left(\frac{2\e d^2}{s} \right) }\Vert \bSigma^{-1}\B\Vert_{2}
			\\
			&\leq 8c_S\sqrt{\frac{2s\log\left(\frac{2\e d^2}{s} \right)}{T\kappa_{\min}} }\Vert \bSigma^{-1}\B X\Vert_{L^2}
			\\
			&\leq 32c^2_S \frac{s\log\left(\frac{2\e d^2}{s} \right)}{T\kappa_{\min}} +\Vert \bSigma^{-1}\B X\Vert^2_{L^2},
		\end{align*}
		and combining these results with \eqref{eq: slope start} completes the proof.
	\end{proof}

As for the Lasso estimator, this leads to results on various norms.
	
	\begin{corollary}\label{cor: slope}
		Let everything be given as in Proposition \ref{prop: slope}, and let $\ep_1\in(0,1)$. 
		Then, for $T>T_0(\ep_1)$ and $T_0(\cdot)$ defined as in \eqref{def:T0},
		the following assertions hold, each with probability larger than $1-\tfrac{1}{2}(\ep_0+\ep_1)$, for all $\ep_0\in(0,1)$:
\leqnomode
	\begin{align}\tag{$\operatorname{a}$}
	\|\bSigma^{-1}(\hat{\A}-\A_0)X\|_{L^2}^2
	&\le32c^2_S \frac{s\log\left(\frac{2\e d^2}{s} \right)}{T\kappa_{\min}}\left(\frac{\log(4\ep_0^{-1})}{s\log\left(\frac{2\e d^2}{s}\right)}\vee 1 \right),\\\label{slope:b} \tag{$\operatorname{b}$}
	\|\bSigma^{-1}(\hat{\A}-\A_0)\|_{2}^2
	&\le64c^2_S \frac{s\log\left(\frac{2\e d^2}{s} \right)}{T\kappa^2_{\min}}\left(\frac{\log(4\ep_0^{-1})}{s\log\left(\frac{2\e d^2}{s}\right)}\vee 1 \right),\\ \tag{$\operatorname{c}$}
	\Vert \bSigma^{-1}(\hat{\A}-\A_0)\Vert_*
	&\le16c_S \frac{s\log\left(\frac{2\e d^2}{s} \right)}{\sqrt{T}\kappa_{\min}}\left(\frac{\log(4\ep_0^{-1})}{s\log\left(\frac{2\e d^2}{s}\right)}\vee 1 \right),\\\label{slope:d}\tag{$\operatorname{d}$}
	\Vert \bSigma^{-1}(\hat{\A}-\A_0)\Vert_1
	&\le16c_S \frac{s\log\left(\frac{2\e d^2}{s} \right)}{\sqrt{T}\kappa_{\min}\log(2)}\left(\frac{\log(4\ep_0^{-1})}{s\log\left(\frac{2\e d^2}{s}\right)}\vee 1 \right).
\end{align}
	\end{corollary}
	\begin{proof}
		The proof is completely analogous to the proof of Corollary \ref{cor: lasso}, except for \eqref{slope:d}, where it suffices to note that $\log(2)\Vert \B\Vert_1\leq \Vert \B\Vert_*$ for all $\B\in\R^{d\times d}.$
	\end{proof}

\subsection{Optimality of the convergence rates}
Alongside the extension of the analysis of the Gaussian OU model to the Lévy-driven case, the principal question of determining \emph{rate-optimal} estimators for high-dimensional models of continuous-time processes is in the focus of our study in this paper.
To mark out the framework, we start by recalling that Theorem 2 in \cite{gama19} provides a minimax lower bound of the form
\[
\inf_{\hat{\A}}\sup_{\A\in\Gamma_s}\E_{\A}\left[\|\hat{\A}-\A\|_2^2\right]\ge c'ds\log(cd/s)/T,
\]
for certain constants $c,c'>0$, where $\Gamma_s$ is the set of row-$s$-sparse matrices and the infimum is taken over all possible estimators of the drift parameter $\A$ in the classical OU model \eqref{eq: SDE cont} with $\bSigma=\operatorname{Id}_{d\times d}$.
In what follows, we will derive a similar result under the sparsity assumption $\|\A\|_0\le s$. 
As regards compatibility of lower and upper bounds, it is of specific advantage that the probability estimates in the previous subsections apply to any confidence level.
	In particular, this allows to prove upper bounds in expectation, conditioned on the event $Q_T(\kappa_{\min}/2)$. 

\begin{corollary}\label{cor:expect}
Grant the assumptions of Proposition \ref{prop: Lasso} and \ref{prop: slope}, respectively, let $\ep_1\in(0,1)$, and recall the definition of the event $Q_T(\cdot)$ in \eqref{def:Q}. 
Then, for $T>T_0(2\ep_1)$ and $T_0(\cdot)$ defined as in \eqref{def:T0},
\begin{align*}
	&\E_{\A_0}\left[\frac{\kappa_{\min}}{2}\Vert \bSigma^{-1}(\hat{\A}_{\mathrm{lasso}}-\A_0)\Vert_{2}^2+2\lambda_T\Vert \bSigma^{-1}(\hat{\A}_{\mathrm{lasso}}-\A_0)\Vert_1 \,\Big\vert\, Q_T\left(\frac{\kappa_{\min}}{2}\right) \right]\\
&\hspace*{18em}\leq\frac{8s\lambda_T^2}{\kappa_{\min}(1-\ep_1)}\left(1+ \frac{2}{\log(2\e d^2)}\right),\\
&\E_{\A_0}\left[\frac{\kappa_{\min}}{2}\Vert\bSigma^{-1}(\hat{\A}_{\mathrm{slope}}-\A_0)X\|_{2}^2+\frac{2c_S}{\log(2)\sqrt{T}} \Vert \bSigma^{-1}(\hat{\A}_{\mathrm{slope}}-\A_0)\Vert_1  \,\Big\vert\, Q_T\left(\frac{\kappa_{\min}}{2}\right) \right]\\
&\hspace*{18em}\leq \frac{32c^2_Ss\log(2\e d^2/s)}{T\kappa_{\min}(1-\ep_1)}\left(1+\frac{2}{\log(2\e d^2)}\right).
\end{align*}	
\end{corollary}

The fact that the above bounds are for the \emph{conditional} expectation, which is in contrast to the results for sparse regression in \cite{belets18}, can be justified by us not \emph{assuming} our property of restricted eigenvalue type, but \emph{proving} that it holds with high probability. 
For this, also note that Proposition \ref{prop: REP} implies that $Q_T(\kappa_{\min}/2)$ is a subset of the event where the restricted eigenvalue type property holds true.

\begin{proof}[Proof of Corollary \ref{cor:expect}]
We start by proving the assertion for the Lasso estimator. 
For now, let $\hat{\A}=\hat{\A}_{\lass}$, and set
\[
Y\coloneqq \left(\Vert \bSigma^{-1}(\hat{\A}_{\mathrm{}}-\A_0)X\Vert_{L^2}^2+2\lambda_T\Vert \bSigma^{-1}(\hat{\A}-\A_0)\Vert_1 \right)\frac{\kappa_{\min}\log(2\e d^2/s)}{8\lambda_T^2}\ \mathds{1}_{Q_T(\kappa_{\min}/2)}.\]
Inspection of the proof of Proposition \ref{prop: Lasso} shows that $Y\leq \log(4\ep_0^{-1})$ holds with probability of at least $1-\ep_0/2$, for all $0<\ep_0\leq\ep_0^*<1$.
 Hence, for any $r\geq r^*= \log(4/\ep_0^*)$,
\[\P_{\A_0}\left(Y>r \right)\leq 2\e^{-r} \]
and therefore
\begin{align*}
\E_{\A_0}[Y]\leq \int_0^\infty \P_{\A_0}(Y>r)\d r\leq r^*+2\int_{r^*}^\infty \e^{-r}\d r\leq r^*+2.
\end{align*}
Choosing $\ep_0^*=4(2\e d^2/s)^{-s}$, we thus obtain
\begin{align*}
	&\E_{\A_0}\left[\left(\Vert \bSigma^{-1}(\hat{\A}-\A_0)X\Vert_{L^2}^2+2\lambda_T\Vert \bSigma^{-1}(\hat{\A}-\A_0)\Vert_1 \right)\ \mathds{1}_{Q_T(\kappa_{\min}/2)}\right]
	\\&\hspace*{3em}\leq\frac{8s\lambda_T^2}{\kappa_{\min}}+ \frac{16s\lambda_T^2}{\kappa_{\min}s\log(2\e d^2/s)}
	\le\frac{8s\lambda_T^2}{\kappa_{\min}}+ \frac{16s\lambda_T^2}{\kappa_{\min}\log(2\e d^2)}.
\end{align*}
Applying Proposition \ref{prop: REP} then yields the assertion for the Lasso estimator.
We continue with the proof for the Slope estimator. 
By an abuse of notation, $\hat{\A}=\hat{\A}_{\slo}$ now denotes the Slope estimator defined as in Proposition \ref{prop: slope}.
Furthermore, let
\[
Z\coloneqq \left(\Vert\bSigma^{-1}(\hat{\A}-\A_0)X\|_{L^2}^2+\frac{2c_S \Vert \bSigma^{-1}(\hat{\A}-\A_0)\Vert_*}{\sqrt{T}} \right) \frac{T\kappa_{\min}}{32c^2_S}\ \mathds{1}_{Q_T(\kappa_{\min}/2)}.
\]
Analogously to the proof for the Lasso estimator, we have that $Z\le\log(4\ep_0^{-1})$ holds with probability of at least $1-\ep_0/2$, for all $0<\ep_0\leq\ep_0^*<1$. 
Hence, choosing $\ep_0^*$ as above,
\begin{align*}
&\E_{\A_0}\left[\left(\Vert\bSigma^{-1}(\hat{\A}-\A_0)X\|_{L^2}^2+\frac{2c_S \Vert \bSigma^{-1}(\hat{\A}-\A_0)\Vert_*}{\sqrt{T}} \right)\ \mathds{1}_{Q_T(\kappa_{\min}/2)} \right]
\\
&\hspace*{3em}\le\frac{32c^2_Ss\log(2\e d^2/s)}{T\kappa_{\min}}\left(1+\frac{2}{s\log(2\e d^2/s)}\right)
\le\frac{32c^2_Ss\log(2\e d^2/s)}{T\kappa_{\min}}\left(1+\frac{2}{\log(2\e d^2)}\right).
\end{align*}
Applying Proposition \ref{prop: REP}, together with $\log(2)\Vert \B\Vert_1\leq \Vert \B\Vert_*$ for all $\B\in\R^{d\times d}$, completes the proof.
\end{proof}

For proving a lower bound for estimation of the drift parameter $\A_0$ over the set of $s$-sparse matrices, belonging to $M_+(\R^d)$, we follow the strategy developed by \cite{belets18} in the  high-dimensional regression setting.
By providing a lower bound on the expected value of a general loss function, one in particular also obtains results allowing for comparison with upper bounds in probability as they are stated in Corollary \ref{cor: lasso} and \ref{cor: slope}, respectively.

\begin{theorem}[cf.~Theorem 7.1 in \cite{belets18}]\label{thm:low}
	Let $d\geq4$, $s\geq 2d$, and consider a nondecreasing function $\ell\colon\R_+\to\R_+$ fulfilling $\ell(0)=0$ and $\ell\not\equiv0$.
	Assume that the Lévy triplet of the BDLP $\Z$ is given by $(0,\operatorname{Id}_{d\times d},0)$.
	Then, for $1\leq p<\infty$, there exist positive constants $c,c'$, depending only on $\ell(\cdot)$, such that
	\[
	\inf_{\hat{\A}}\sup_{\A_0\in M_+(\R^d)\cap \mathbb{B}_0(s)}\E_{\A_0}\left[\ell\left(c\psi^{-1}_{T,p} \|\hat \A-\A_0\|_{p}\right)\right]\ge c',
	\]
	where the infimum extends over all estimators of $\A_0$ and
	\[\psi_{T,p}\coloneq s^{1/p} \sqrt{\frac{\log(\e d^2/s)}{T}}.\]
\end{theorem}
\begin{proof}
Let $r$ be the largest even number such that $r\leq (s-d)/2$, and let $\Omega_r$ be the set of antisymmetric matrices in $\{-1,0,1\}^{d\times d}$ with sparsity exactly equal to $r$. 
Then, every matrix in $\Omega_r$ is uniquely determined by its upper triangular section, which corresponds to a vector in $\{-1,0,1\}^{d(d-1)/2}$ with $r/2$ non-zero entries. 
Now since $d\geq 4$ and $s\geq 2d$ imply $d(d-1)/2\geq2$ and $1\leq r/2\leq (s-d)/4\leq d(d-1)/4$, Lemma F.1 in \cite{belets18} entails the existence of a set $\tilde{\Omega}_r\subseteq\Omega_r$ such that, for all $\B\neq\B'\in\tilde{\Omega}_r$,
\begin{align}\label{eq: sparsity}
\Vert \B\Vert_0&\leq r\le s-d,\\\nonumber
\log\left(\vert \tilde{\Omega}_r\vert\right)&\ge cr\log(\e d(d-1)/r)\geq cr\log(\e d^2/s), \\\label{eq: hamming}
\Vert \B-\B'\Vert^p_p&\geq r/8\geq ((s-d)/2-1)/8\ge s/64,\quad p\ge 1,
\end{align}
where $c>0$ is an absolute constant and we used $d\geq4$ and $s\geq 2d$ for \eqref{eq: hamming}.  
Now, for $w>0$, set 
\[\Omega_w\coloneq\left\{\tfrac{1}{2}\operatorname{Id}_{d\times d}+w\B:\B\in\tilde{\Omega}_r\right\}. \]
Note that $\mathrm{i} \B$ is unitarily diagonizable for every $\B\in \tilde{\Omega}$, because of its antisymmetricity. 
Hence, for any $\A\in\Omega_w$, there exist a unitary matrix $\mathbf{U}$ and a real diagonal matrix $\mathbf{D}$ such that, for $t>0$,
\begin{align*}
\Vert \e^{-t\A }\Vert_{\mathrm{Sp}}\leq \e^{-\frac t 2} \Vert\mathbf{U} \e^{-\mathrm{i}tw\mathbf{D}}\mathbf{U}^*\Vert_{\mathrm{Sp}}\leq  \e^{-\frac t 2} \Vert \e^{-\mathrm{i}tw\mathbf{D}}\Vert_{\mathrm{Sp}}=\e^{-\frac t 2}.
\end{align*}
Thus, $\Vert \e^{-t\A}\Vert_{\mathrm{Sp}}\to 0$ as $t\to\infty$ holds for any $\A\in\Omega_w$, implying that $\Omega_w\subset M_+(\R^d)$.
Furthermore, by \eqref{eq: sparsity}, $\Vert \A\Vert_0\leq s$ holds for all $\A\in\Omega_w$.
Lemma 6 in \cite{gama19} entails that, under $\P_{\A}$, 
\begin{align*}
\C_\infty=\operatorname{Id}_{d\times d},
\end{align*}
and \eqref{eq: sparsity} gives for $\A,\A'\in \Omega_w$ that
\[\Vert \A-\A'\Vert^2_{2}\leq 4w^2r. \]
For the Kullback--Leibler divergence of the probability measures associated to $\A,\A'\in\Omega_w$, we then get as in the proof of Corollary 3 in \cite{gama19}
\[\operatorname{KL}(\P_{\A}\Vert \P_{\A'})=\frac{T}{2}\tr((\A'-\A)\C_\infty(\A'-\A)^\top )\leq 2Tw^2 r, \]
and \eqref{eq: hamming} implies for $p\geq 1$
\[\Vert \A-\A'\Vert_p^p\geq sw^p/64.\] 
Now, choosing $w_0>0$ such that $w_0^2= c T^{-1}\log(\e d^2/s)/25$ and setting $\Omega=\Omega_{w_0}$, it holds for all $\A,\A'\in\Omega$
\begin{align*}
\operatorname{KL}(\P_{\A}\Vert \P_{\A'})&\leq  2cr \log(\e d^2/s)/25< \log\left(\vert \Omega\vert\right)/8,
\\
\Vert \A-\A'\Vert_p^p&\geq s (c T^{-1}\log(\e d^2/s))^{p/2}/320,
\end{align*}
which completes the proof by applying Theorem 2.7 in \cite{tsybakov09}.
\end{proof}

Using the indicator loss $\ell(u)=\mathds{1}\{u\ge 1\}$, Theorem \ref{thm:low} yields, e.g., the following statement for $d\geq 4$ and $s\geq 2d$: 
For any estimator $\hat{\A}$, there exists some $s$-sparse matrix $\A_0\in M_+(\R^d)$ such that, with $\P^{\A_0}$-probability of at least $c_0$,
\[
 \|\hat \A-\A_0\|_2\ge c_1\sqrt{\frac{s\log(\e d^2/s)}{T}},
\]
for some constants $c_0,c_1>0$.
Note that this lower bound matches the upper bound for the Slope estimator of the drift parameter $\A_0$ in the model \eqref{eq: sde levy} with $\bSigma=\mathrm{Id}_{d\times d}$ which was derived in Corollary \ref{cor: slope}\eqref{slope:b}.
The restrictions $d\geq4$ and $s\geq 2d$ are consequences of the assumption $s\in[1,p/2]$ in Lemma F.1 of \cite{belets18} and the construction of the hypotheses. 
More specifically, as we want to apply Lemma 6 in \cite{gama19} for showing that $\C_\infty$ is identical for all hypotheses, we use a similar construction by antisymmetric matrices as in Lemma 5 of the same reference. 
However, as the constructed set then needs to contain matrices with sparsity $\leq s-d$, we are in need of a lower bound of the form $s-d\geq cs$ for some constant $c>0$ which holds for \textit{all} $d\geq 4$. 
This is reflected in \eqref{eq: hamming} in the proof of Theorem \ref{thm:low}, where it can also be seen that the assumption $s\geq 2d$ solves this problem.

\section{Deviation inequalities} \label{sec: dev ineq}
Having presented our main results for the Lasso and Slope estimator, respectively, in the last section, we now give the two central deviation inequalities used in the proofs. 
In particular, the approach to bounding the stochastic error introduced in Section \ref{sec: dev chaining} provides the key to achieving the optimal rate of convergence.

 \subsection{Property of restricted eigenvalue type}\label{sec: ret prop}
In previous works on Lasso and Slope estimators, the analysis relied on the so-called restricted eigenvalue property, which in our setting corresponds to the assumption
 \[\inf_{\B\in\mathcal{C}} \frac{\Vert \B X\Vert_{L^2}^2}{\Vert \B\Vert_{2}^2}\geq c_{\mathrm{REP}}, \]
for certain cones $\mathcal{C} \subset \R^{d\times d}$ and a constant $c_{\mathrm{REP}}>0$. 
It was discovered in \cite{gama19} and \cite{cmp20} that this property holds with high probability in the context of Lasso estimation for Gaussian OU processes as soon as assumptions corresponding to \ref{ass: concentration} are fulfilled (see Theorem 3 in \cite{gama19} and Theorem 3.3 in \cite{cmp20}). 
As these findings essentially rely on the discretization procedure presented in Lemma F.2 of \cite{basu2015}, it is not surprising that similar results can be obtained in the Lévy-driven case. 

\begin{proposition}\label{prop: REP}
	For any $T>0$, it holds
	\[	Q_T\left(\frac{\kappa_{\min}}{2} \right)\subseteq	\left\{\inf_{\B\in\R^{d\times d}\setminus\{0\} }\frac{\Vert \B X\Vert_{L^2}^2}{\Vert \B\Vert_{2}^2}\geq \frac{\kappa_{\min}}{2} \right\}
 \]
and, for any $r>0$, we have
		\begin{equation}\label{ineq:Q}
			\P\left(Q_T(r) \right)\geq1-(21(d\land\e))^dH\left(T,\frac{r}{3d}\right). 
		\end{equation}	
	\end{proposition}
	\begin{proof}
First note that
		\begin{align*}
			\left\{\inf_{\B\in\R^{d\times d}\setminus\{0\} }\frac{\Vert \B X\Vert_{L^2}^2}{\Vert \B\Vert_{2}^2}< \frac{\kappa_{\min}}{2} \right\}
			&= \left\{\kappa_{\min}-\inf_{\B\in\R^{d\times d}\setminus\{0\} }\frac{\tr( \B \hat{\C}_T\B^\top)}{\Vert \B\Vert_{2}^2}> \frac{\kappa_{\min}}{2} \right\}
			\\
			&= \left\{\inf_{\B\in\R^{d\times d}\setminus\{0\} }\frac{\tr( \B \C_\infty\B^\top)}{\Vert \B\Vert_{2}^2}-\inf_{\B\in\R^{d\times d}\backslash\{0\} }\frac{\tr( \B \hat{\C}_T\B^\top)}{\Vert \B\Vert_{2}^2}> \frac{\kappa_{\min}}{2} \right\}
			\\
			&\subseteq \left\{\sup_{\B\in  \mathbb{B}_2(1)}\vert\tr( \B (\hat{\C}_T-\C_\infty)\B^\top)\vert> \frac{\kappa_{\min}}{2} \right\},
		\end{align*}
		where we used the identities $\Vert \B X\Vert_{L^2}^2=\tr( \B \hat{\C}_T\B^\top)$ and 
		\[	\tr( \B \C_\infty\B^\top)=\vec(\B^\top)^\top(\operatorname{Id}_{d\times d}\otimes\C_\infty)\vec(\B^\top),
		\]
		with $\otimes$ denoting the Kronecker product, combined with the min-max theorem. This proves the first assertion.
		Since, for $\B\in\R^{d\times d}$,
		\begin{equation}\tr(\B (\hat{\C}_t-\C_\infty)\B^\top)=\sum_{i=1}^d \B_{i,\cdot}(\hat{\C}_t-\C_\infty)\B_{i,\cdot}^\top ,\label{eq: structure rep}
		\end{equation} where $\B_{i,\cdot}$ denotes the $i$-th row of $\B$, assumption \ref{ass: concentration} now implies together with Lemma 7 in \cite{gama19} for any $r>0$
		\begin{align*}
			\P\left(\sup_{\B\in\mathbb{B}_2(1) }\vert \tr(\B (\hat{\C}_T-\C_\infty)\B^\top)\vert\geq r \right)
			&\le \P\left(\sup_{u \in \R^d: \Vert u\Vert\leq 1 }\vert  u^\top(\hat{\C}_T-\C_\infty)u\vert\geq \frac{r}{d} \right)\\
			&\leq (21(d\land\e))^dH\left(T,\frac{r}{3d}\right),
		\end{align*}	
	i.e., \eqref{ineq:Q} holds.
	\end{proof}

As can be seen in Proposition \ref{prop: REP}, we choose $\mathcal{C}=\R^{d\times d}\setminus \{\boldsymbol{0}\}$.
This may seem counterintuitive at first, since the cones used in previous works are much smaller than the whole space. 
There are two main reasons for our choice. 
Firstly, we employ Proposition \ref{prop: REP} in Section \ref{sec: dev chaining} for obtaining a deviation inequality for the stochastic error term involving $\boldsymbol{\ep}_T$ in Lemma \ref{lemma: L2 frobenius}.
 As this deviation inequality must hold for all matrices, we have to choose $\mathcal{C}$ in the specified way. 
 Secondly, as our framework concerns sparsity instead of row-sparsity, it becomes hard to exploit the property \eqref{eq: structure rep}. 
 A good indicator for this is the difference between the threshold time index $T_0$ in Theorem 3.3 of \cite{cmp20} and Corollary 4 in \cite{gama19}: 
 In the row-sparse setting, the dominating term wrt sparsity and dimension is of the form $s\log(d/s)$, whereas in the sparse setting the corresponding term is given as $s\log(d)$, which is clearly larger since the sparsity always dominates the row-sparsity. 
 This is in fact a direct consequence of the different concentration results stated in Lemma 6.2 in \cite{cmp20} and Lemma 8 in \cite{gama19}, which can be seen as analogues to Proposition \ref{prop: REP}.
 Recall that \cite{cmp20} assumes the true parameter to be sparse whereas in \cite{gama19} the parameter is assumed to be row-sparse.

\subsection{Bounding the stochastic error}\label{sec: dev chaining}
We now prove a uniform deviation inequality for the stochastic error term involving $\boldsymbol{\ep}_T$ in the basic inequality stated in Lemma \ref{lemma: L2 frobenius}. 
As we want to obtain results for the Slope estimator, we are in need of a statement similar to Theorem 4.1 in \cite{belets18}. 
However, the proof of said theorem strongly relies on the noise being normally distributed, as it uses as a key argument the classical concentration result for Lipschitz functions of Gaussian random variables (see, e.g., Theorem 5.2.2 in \cite{vers18}).
Since the noise in our case is given by an It\={o} integral, we are not able to directly employ the same techniques as used in \cite{belets18}. 
We overcome this challenge by noting that Proposition \ref{prop: REP} allows us to find a uniform bound for the quadratic variation of the noise term, which holds with high probability. 
This implies that the noise is sub-Gaussian with high probability, thus enabling us to apply Talagrand's generic chaining device and majorizing measure theorem (see e.g.\ Chapter 2 in \cite{tala14} or Section 8.6 in \cite{vers18}) to return to the Gaussian setting. 
These findings yield the following important auxiliary result, for which we define the Gaussian width $w(\mathcal{D})$ and radius $\mathrm{rad}(\mathcal{D})$ of a set $\mathcal{D}\subset \R^{d\times d}$ by setting
\[w(\mathcal{D})\coloneqq\E\left[\sup_{\B\in \mathcal{D}} \qv{\vec(\B)	, Z} \right]\quad\text{ and }\quad \operatorname{rad}(\mathcal{D}) \coloneqq \sup_{\B\in \mathcal{D}}\Vert \B\Vert_{2},  \]
where $Z\sim \mathcal{N}(0,\operatorname{Id}_{d^2\times d^2})$.
Recall the definition of $\boldsymbol \ep_T$ in \eqref{def:ep}.

	\begin{lemma}\label{lemma: chaining}
	 There exists a universal constant $c_0>0$ such that, for any $\mathcal{D}\subset \R^{d\times d}$, it holds for all $u>0$
		\begin{align*}
			\P\left(\sup_{\B\in \mathcal{D}}\qv{\boldsymbol{\ep}_T,\B}_{2}\ \mathds{1}_{Q_T(\kappa_{\max})}\leq c_0 \sqrt{\frac{12\kappa_{\max}}{T}}(w(\mathcal{D})+u \operatorname{rad}(\mathcal{D}))\right)
			\geq 1-2\exp(-u^2).
		\end{align*}
	\end{lemma}
	\begin{proof}
		Note first that, by the min-max theorem, for all $\B\in\R^{d\times d}$,
		\begin{align*}
			\tr(\B\C_\infty\B^\top)=\vec(\B^\top)^\top(\operatorname{Id}_{d\times d}\otimes\C_\infty)\vec(\B^\top)
			\leq\kappa_{\max}\Vert \B\Vert_{2}^2,
		\end{align*}
	and hence
	\begin{align*}
			Q_T(\kappa_{\max})&=\left\{\sup_{\B\in\mathbb{B}_2(1)} \tr(\B(\hat{\C}_T-\C_\infty)\B^\top)\leq  \kappa_{\max} \right\}
\\&\subseteq \left\{\forall \B_1,\B_2\in\mathcal{D}: \tr((\B_1-\B_2)(\hat{\C}_T-\C_\infty)(\B_1-\B_2)^\top)<  \kappa_{\max}\Vert \B_1-\B_2\Vert_{2}^2 \right\}
\\&\subseteq \left\{\forall \B_1,\B_2\in\mathcal{D}:\tr((\B_1-\B_2)\hat{\C}_T(\B_1-\B_2)^\top)<2\kappa_{\max}\Vert \B_1-\B_2\Vert_{2}^2 \right\}.
	\end{align*}
		Let $\B_1\neq\B_2\in\mathcal{D}$ be given. 
		Then, using Bernstein's inequality for continuous martingales, we get
		\begin{align*}
			&\E\left[\exp\left(\frac{\qv{\boldsymbol{\ep}_T,\B_1-\B_2}_{2}^2\ \mathds{1}_{Q_T(\kappa_{\max})}}{12T^{-1}\kappa_{\max}\Vert \B_1-\B_2\Vert_{2}^2}\right)\right]
			\\&\hspace*{3em}\leq \int_0^\infty \P\left(\exp\left(\frac{\qv{\boldsymbol{\ep}_T,\B_1-\B_2}_{2}^2\ \mathds{1}_{Q_T(\kappa_{\max})}}{12T^{-1}\kappa_{\max}\Vert \B_1-\B_2\Vert_{2}^2}\right)>u \right)\d u
			\\&\hspace*{3em} \leq 1+ \int_1^\infty \P\left(\vert \qv{\boldsymbol{\ep}_T,\B_1-\B_2}\vert\ \mathds{1}_{Q_T(\kappa_{\max})}>\sqrt{12T^{-1}\kappa_{\max}\log(u)}\Vert \B_1-\B_2\Vert_{2} \right)\d u
			\\&\hspace*{3em}=1+ \int_1^\infty \P\left(\Big\vert\int_0^T((\B_1-\B_2) X_s)^\top \d W_s\Big\vert>\sqrt{12T\kappa_{\max}\log(u)}\Vert \B_1-\B_2\Vert_{2},Q_T(\kappa_{\max})  \right)\d u
			\\&\hspace*{3em}\leq 1+ \int_1^\infty \P\Bigg(\Big\vert\int_0^T((\B_1-\B_2) X_s)^\top \d W_s\Big\vert>\sqrt{12T\kappa_{\max}\log(u)}\Vert \B_1-\B_2\Vert_{2},
			\\&\hspace{5cm}\int_0^T \Vert(\B_1-\B_2)  X_s\Vert^2\d s<2T\kappa_{\max}\Vert \B_1-\B_2\Vert_{2}^2  \Bigg)\d u
			\\&\hspace*{3em}\leq1+2 \int_1^\infty u^{-3}\d u=2.
		\end{align*}
	This shows that $(\qv{\boldsymbol{\ep}_T,\B}_{2}\ \mathds{1}_{Q_T(\kappa_{\max})})_{\B\in\mathcal{D}}$ is sub-Gaussian in the sense of Definition 2.5.6 in \cite{vers18}, since \[\Vert\qv{\boldsymbol{\ep}_T,\B_1}_{2}-\qv{\boldsymbol{\ep}_T,\B_2}_{2}\Vert^2_{\psi_2}\leq \frac{12}{T}\kappa_{\max}\Vert\B_1-\B_2\Vert_2^2,  \]
	holds, where $\Vert Y\Vert_{\psi_2 }\coloneq\inf\{t>0: \E[\exp(Y^2/t^2)\leq 2]\},$ for any random variable $Y$.
	Hence, we can apply Exercise 8.6.5 in \cite{vers18}, which yields for any $\mathcal{D}\subset \R^{d\times d}$ and for all $u>0$ the asserted inequality.
	\end{proof}
Combining Lemma \ref{lemma: chaining} with the concentration property for Lipschitz functions of Gaussian random variables and Proposition E.2 in \cite{belets18}, which allow us to bound the Gaussian width of the relevant set, we arrive at the following proposition. 
	
	\begin{proposition}\label{prop: chain slope}
		 For $0<\ep_0<1$, set 
		 	\begin{equation}\label{def:norms}
		 	\Vert  \B\Vert_S\coloneq \Vert \B\Vert_*\vee\sqrt{\log(4\ep_0^{-1})} \Vert \B\Vert_{2},\quad \forall \B\in\R^{d\times d}.
		 	\end{equation}
Then, for all $T>0$,
		\begin{align*}
			\P\left(\sup_{\B\in \R^{d\times d}, \B\neq 0}\frac{\qv{\boldsymbol{\ep}_T,\B}_{2}}{\Vert \B\Vert_S}
		\	\mathds{1}_{Q_T(\kappa_{\max})}\leq c_* \sqrt{\frac{\kappa_{\max}}{T}}\right)\geq 1- \frac{\ep_0}{2},
		\end{align*} 
		where, for $c_0$ being the constant from Lemma \ref{lemma: chaining},
		\[
		c_*\coloneq c_0\left(\sqrt{\frac{3\pi}{\log(2)}}+\sqrt{300}\right).	
		\]
	\end{proposition}
	\begin{proof}
		First note that
		\[\sup_{\B\in \R^{d\times d}, \B\neq 0}\frac{\qv{\boldsymbol{\ep}_T,\B}_{2}}{\Vert \B\Vert_S}=\sup_{\B\in \mathcal{D}_*}\qv{\boldsymbol{\ep}_T,\B}_{2},\quad\text{ where }
		\mathcal{D}_*\coloneq \left\{\B\in\R^{d\times d}: \Vert \B\Vert_S= 1\right\}.
		\] 
		Thus, to apply Lemma \ref{lemma: chaining}, we need to bound $w(\mathcal{D}_*)$ and $\operatorname{rad}(\mathcal{D}_*)$.
		Therefore, let $Z\sim \mathcal{N}(0,\operatorname{Id}_{d^2\times d^2})$, and note
		that the function 
		\[
		f\colon \R^{d^2} \to \R, \quad v\mapsto f(v)\coloneq \sup_{\B\in \mathcal{D}_*} \qv{\vec(\B),v},
		\]  
		is Lipschitz continuous with Lipschitz constant $\log(4)^{-1/2}$ wrt the Euclidean distance. 
		Thus, equation (1.4) in \cite{tala91} gives
		\begin{align*}
			w(\mathcal{D}_*)=\E[f(Z)]
			&\leq \int_0^\infty \P\left(\vert f(Z)-\operatorname{Med}(f(Z))\vert \geq u \right)\d u +\operatorname{Med}(f(Z))
			\\
			&\leq \int_0^\infty \exp\left(- \frac{\log(4)u^2}{2} \right)\d u +\operatorname{Med}(f(Z))
			\\
			&=\sqrt{\frac{\pi}{4\log(2)}}+\operatorname{Med}(f(Z)).
		\end{align*}
		Combining Proposition E.2 in \cite{belets18} with
		\begin{align*}
			f(Z)&\leq\sup_{\B\in \R^{d\times d}:\Vert \B\Vert_*\leq 1} \qv{\vec(\B)	, Z} 
			\\
			&=\sup_{\B\in \R^{d\times d}:\Vert \B\Vert_*\leq 1} \sum_{i=1}^{d^2}\vec(\B)^\#_i\sqrt{\log(2d^2/i)} \frac{Z^\#_i}{\sqrt{\log(2d^2/i)}} 
			\\
			&\leq \max_{i=1,\ldots,d^2} \frac{Z^\#_i}{\sqrt{\log(2d^2/i)}} 
		\end{align*}
		yields $\operatorname{Med}(f(Z))\leq 4$, implying that
		\[w(\mathcal{D}_*)\leq \sqrt{\frac{\pi}{4\log(2)}}+4.\] 
		Since $\operatorname{rad}(\mathcal{D}_*)\leq \log(4/\ep_0)^{-1/2}$ trivially holds, the assertion follows.
	\end{proof}

\section{Discussion of assumption \ref{ass: concentration} and outlook}
	\subsection{Sufficient conditions for assumption \ref{ass: concentration}} \label{sec: ass conc}
	We first recall the results of \cite{gama19} and \cite{cmp20} on assumption \ref{ass: concentration} for the case where the BDLP is given as a standard Wiener process.
	Moreover, we prove that in the Lévy-driven case \ref{ass: concentration} is satisfied as soon as the Lévy measure of the BDLP admits a fourth moment.
	
	\paragraph{The Gaussian case}
	As mentioned above, both \cite{gama19} and \cite{cmp20} assume that $\Z$ is a standard Wiener process, i.e., the characteristic triplet of $\Z$ is given by $(0,\operatorname{Id}_{d\times d}, 0)$. 
	In this case, \cite{gama19} were able to show that \ref{ass: concentration} holds under assumptions implied by \ref{ass: ergodicity} if $\A_0$ is symmetric. 
	This result was achieved by exploiting that symmetricity of $\A_0$ implies $\mu$ to fulfill a $\log$-Sobolev inequality, which then yields \ref{ass: concentration} by Theorem 2.1 of \cite{cattiaux08}. 
	\cite{cmp20} extended this finding to the general case of possibly non-symmetric $\A_0$, i.e., $\A_0\in M_+(\R^d)$ already  implies \ref{ass: concentration} in the classical Gaussian case. 
	The proof of this result relies on Malliavin calculus methods, especially Theorem 4.1 in  \cite{nourdin09}. 
	In both papers, the function $H$ in \ref{ass: concentration} is of the form
	\[H(T,r)=2\exp(-T H_0(r)), \quad T,r>0, \]
	where $H_0$ is positive and increasing. 
	For the sake of completeness, we state the findings of \cite{cmp20} below.
	
	\begin{proposition}[cf. Proposition 3.2 in \cite{cmp20}] \label{prop: gauss conc}
Assume that the characteristic triplet of the BDLP $\Z$ is given by $(0,\operatorname{Id}_{d\times d},0)$. 
Denote by $\lambda_1,\ldots,\lambda_d$ the eigenvalues of $\A_0$, and let $\mathbf{P}_0$ be  the matrix such that $\A_0=\mathbf{P}_0 \operatorname{diag}(\lambda_1,\ldots,\lambda_d)\mathbf{P}_0^{-1}$.
Then, for all $r>0$,
\[\sup_{u\in\R^d: \Vert u\Vert=1}\P\left(\vert u^\top(\hat{\C}_T-\C_\infty)u\vert\geq r \right)\leq 2\exp(-T H_0(r)), \]
where 
\[H_0(r)=\frac{\tau_0r^2}{8\lambda_{\max}(\C_\infty)p_0\left(r+\lambda_{\max}(\C_\infty)\right)},\quad r>0, \]
with $\tau_0=\min(\Re(\lambda_i))$ and $p_0=\Vert \mathbf{P}_0\Vert_{\mathrm{Sp}}\Vert \mathbf{P}^{-1}_0\Vert_{\mathrm{Sp}}$.
	\end{proposition}

	\paragraph{The Lévy-driven case}
	Since the derivation of the results in the previous paragraph strongly relies on the Gaussianity of $\X$, achieving similar results in the Lévy-driven setting is a challenging task. 
	However, an application of the stochastic Fubini theorem (similar to the proof of equation (2.17) in \cite{barndorff97}), combined with classical martingale results, yields that \ref{ass: concentration} is fulfilled as soon as the Lévy measure $\nu$ of the BDLP admits a fourth moment.
	
	\begin{proposition}\label{thm: levy conc}
		Assume that $\nu$ admits a fourth moment. 
		Then, there exists a constant $c>0$ such that, for all $u\in\R^d$ fulfilling $\Vert u\Vert\leq 1$,
		\begin{align*}
			\P\left(\vert u^\top (\hat{\C}_T-\C_\infty)u\vert \geq r \right)
			\leq \frac{c}{t(r\land r^2)}+ \frac{c}{ (tr)^2}.
		\end{align*}
		In particular, Assumption \ref{ass: concentration} is fulfilled.
	\end{proposition}
The proof of Proposition \ref{thm: levy conc}, which also contains the explicit value of the constant $c$, can be found in Appendix \ref{app: proof conc}.

\subsection{Outlook}
	Following the pioneering work of \cite{gama19} and \cite{cmp20} which clarified the statistical foundations of a high-dimensional modelling of the classical OU process, we have extended the investigation to the Lévy-driven case. 
	In particular, this requires finding tools that do not explicitly rely on Gaussian structures.
	
	As usual in high-dimensional statistics, the proof of our main results (Propositions \ref{prop: Lasso} and \ref{prop: slope}) is based on two central elements: 
	On the one hand, we confine ourselves to the study of a benign event, in our context of the form
	\[
	\mathscr E\coloneqq \left\{\inf_{\B\in\R^{d\times d}\backslash\{0\} }\frac{\Vert \B X\Vert_{L^2}^2}{\Vert \B\Vert_{2}^2}> \frac{\kappa_{\min}}{2}\right\}\bigcap \left\{\sup_{\B\in \R^{d\times d}: \B\neq 0}\frac{\qv{\boldsymbol{\ep}_T,\B}_{2}}{\Vert \B\Vert_S}\leq c_* \sqrt{\frac{\kappa_{\max}}{T}}\right\}.
	\]
	As becomes clear in the proof of the aforementioned propositions, the investigation on this event is driven by purely deterministic arguments, which can be developed analogously to the high-dimensional linear regression model as it is studied in \cite{belets18}.
	It then remains to show that the event $\mathscr E$ is of high probability.
	
	With respect to the first sub-event, this amounts to verifying a property of restricted eigenvalue type.
	Similarly to the Gaussian case, we identified a concentration condition (assumption \ref{ass: concentration}) that can be used to show this. 
	Proposition \ref{thm: levy conc} stated in the previous subsection gives a concrete criterion for \ref{ass: concentration} to be fulfilled.
	 This result is obviously weaker than its Gaussian counterpart (Proposition \ref{prop: gauss conc}) in the sense of the temporal decay not being exponential but polynomial.
	 The primary influence of this is on the value of the threshold value $T_0$ specified in \eqref{def:T0} appearing in Corollaries \ref{cor: lasso} and \ref{cor: slope}, which increases.
	 Nevertheless, as the main results of this paper are developed in such a way that they only rely on assumption  \ref{ass: concentration} in its general form, it would be easy to implement results implying an exponential decay in the Lévy-driven case to achieve values of $T_0$ similar to the Gaussian case.
	 
	 The second sub-event of $\mathscr E$ involves both the process $\boldsymbol{\ep}_T$ (specified in \eqref{def:ep}) and the norm $\Vert\cdot\Vert_S$ (as introduced in \eqref{def:norms}).
	 At this point, the main differences with the studies of \cite{gama19} and \cite{cmp20} do not arise because of the structure of the process, but because of the different statistical approach. 
	 In fact, controlling the second sub-event provides the key to removing the additional logarithmic factor in the convergence rate.
	The derivations in Section \ref{sec: dev chaining} are therefore of independent interest. 
	As noted in Remark 4.4 of \cite{cmp20}, the development of general high-dimensional diffusion models requires a suitable representation of the likelihood function (given in our case by Proposition \ref{prop: sorensen likelihood}) and appropriate techniques for proving concentration phenomena. 
	If these ingredients are combined with our techniques for bounding the stochastic error, estimators (of the Lasso or Slope type) that achieve minimax optimal convergence rates might also be formulated in a general diffusion model.

\section{Simulation study}\label{sec: sim}
\begin{figure}[htp]	
	\centering
	\includegraphics[scale=0.7]{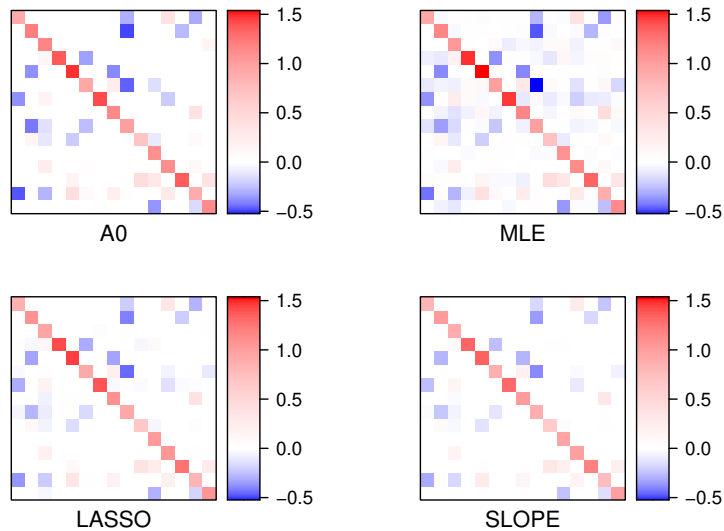}
	\caption{\small Comparison of the true parameter $\A_0$ with MLE, Lasso and Slope.}
	\label{fig: example}
\end{figure}

In this section, we investigate our theoretical results by applying them to simulated data. 
For this purpose, we compare the errors of the maximum likelihood, Lasso and Slope estimators in different dimensions. 
Of course, our results were derived in the setting of continuous observations, but they can easily be transferred to the more realistic framework of discrete observations by discretising the integrals involved. 
The data will always be generated by an Euler--Maruyama scheme with step size $\delta=10^{-2}$. 
We choose this value for $\delta$ because Figure 6 in \cite{gama19} indicates that the quality of estimation does not improve with a smaller step size. 
Since it is well known that choosing tuning parameters by theoretical results leads to too large values, we select the tuning parameters by cross-validation, with the first $80\%$ of the path acting as the training set and the remainder as the validation set. 
More precisely, we define a candidate set $\Lambda\subset \R_+$, and for each $\lambda\in\Lambda$, we set
\[\hat{\A}^{\mathrm{lasso}}_\lambda=\operatorname{arg min}_{\A\in\R^{d\times d}} \mathcal{L}_{[0,0.8T]}(\A)+\lambda \Vert \A\Vert_1,\quad \hat{\A}^{\mathrm{slope}}_\lambda=\operatorname{arg min}_{\A\in\R^{d\times d}} \mathcal{L}_{[0,0.8T]}(\A)+\lambda \Vert \A\Vert_* \]
and
\[\hat{\lambda}^{\mathrm{lasso}}=\operatorname{arg min}_{\lambda \in\Lambda}\frac{\mathcal{L}_{[0.8T,T]}(\hat{\A}^{\mathrm{lasso}}_\lambda)}{\Vert \hat{\A}^{\mathrm{lasso}}_\lambda\Vert_1} ,\quad \hat{\lambda}^{\mathrm{slope}}=\operatorname{arg min}_{\lambda \in\Lambda}\frac{\mathcal{L}_{[0.8T,T]}(\hat{\A}^{\mathrm{slope}}_\lambda)}{\Vert\hat{\A}^{\mathrm{slope}}_\lambda \Vert_*},\]
where $\mathcal{L}_{[0,0.8T]}$ and $\mathcal{L}_{[0.8T,T]}$, respectively, correspond to the negative $\log$-likelihood function computed on the relevant intervals.
This then leads to $\hat{\A}^{\lass}_{\hat{\lambda}^{\lass}}$ and $\hat{\A}^{\slo}_{\hat{\lambda}^{\slo}}$ as our final estimators. 
The candidate set will always be a logarithmic grid with values between $10^{-3}$ and $10$. 
We choose this particular form of our estimators because it is closer to practice compared to the theoretical definitions in \eqref{def:lasso} and \eqref{def: slope}.
For a more in-depth numerical analysis in the Gaussian framework and, in particular, an application to real world financial data, we refer to Section 4 of \cite{gama19}, and for a comparison between Lasso and Dantzig estimators to Section 5 of \cite{cmp20}. 

\begin{figure}[H]
	\centering
\includegraphics[scale=0.7,trim= 0 0 0 55,clip]{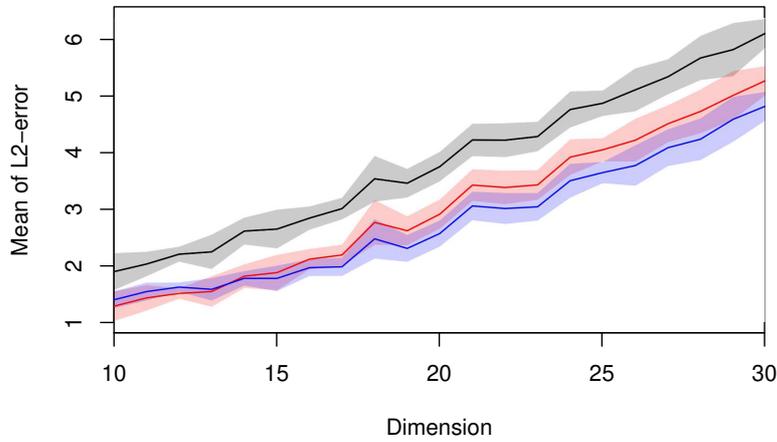}
	\caption{\small$L_2$ errors of MLE, \textcolor{red}{Lasso} and \textcolor{blue}{Slope} $\pm$ one standard deviation.
	}
	\label{fig: comp l2}
\end{figure}

In Figure \ref{fig: example}, we give a first example of the different estimators compared to the ground truth $\A_0$, which in this case is given as a $15\times15$ matrix with sparsity $\sim 0.2$. 
For comparability, we depict the matrices as heat maps. 
In this example, we set $T=300$ and let the matrix $\bSigma$ be generated as a diagonal matrix with entries generated from a uniform distribution with values in $[0,5]$, and the jumps are given by a composite Poisson process with intensity $10$ and Laplace-distributed jump sizes. 
 We choose $\bSigma$ as the diagonal matrix because our results rely on the sparsity of $\bSigma^{-1}\A_0$ and this is the simplest way to preserve the sparsity of $\A_0$. 
 
\begin{figure}[H]
	\centering
\includegraphics[scale=0.7,trim= 0 0 0 55,clip]{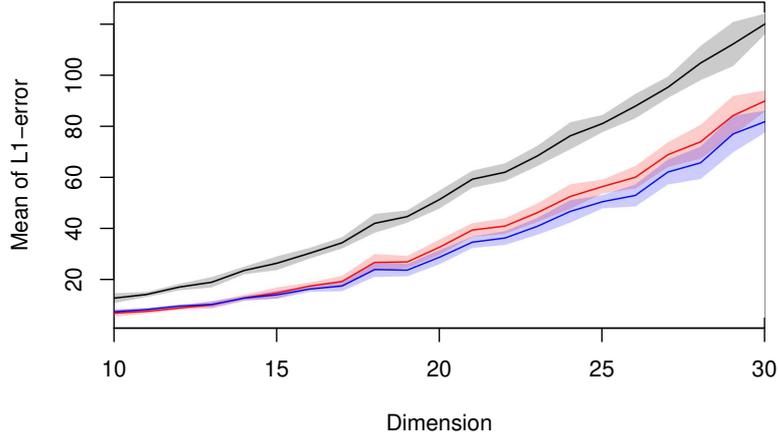}
	\caption{\small $L_1$ errors of MLE, \textcolor{red}{Lasso} and \textcolor{blue}{Slope} $\pm$ one standard deviation.
	}
	\label{fig: comp l1}
\end{figure}
For more general results, we compare the estimation error in $L_1$ and Frobenius norm (hereafter referred to as $L_2$ norm) of the three estimators over $10$ iterations for dimensions $10$ to $30$, with $T=100$. 
For each dimension, we generate $\A_0\in M_+$ with sparsity $\sim 0.2$ and $\bSigma$ similar to Figure \ref{fig: example}, with the only difference that the uniform distribution is now on $[0,10]$. 
The jump intensity is given as $5$, and the jump sizes are also Laplace distributed. 
The results of this simulation study can be seen in Figures \ref{fig: comp l2} and \ref{fig: comp l1}.
We see that Lasso and Slope constantly outperform the maximum likelihood estimator for both error measures, and that Lasso and Slope behave very similarly, which is in line with our theoretical results. 
Moreover, the $L_2$ error grows linearly, while the growth of the $L_2$ error is of quadratic nature.
This also matches our theoretical results.

\appendix
\section{Some results on Lévy processes and Lévy-driven OU processes}\label{app: levy facts}
	 We start by presenting some results on Lévy-driven OU processes and Lévy processes, respectively infinitely divisible distributions. 
	 For this, recall that a function $g\colon \R^d \to\R$ is called sub-multiplicative if it is nonnegative and there exists a constant $c>0$ such that
\[g(x+y)\leq cg(x)g(y),\quad \forall x,y\in\R^d. \]

\begin{lemma}[cf. Theorem 25.3 in \cite{sato99}]\label{lemma: levy proc moment}
	Let $\Z=(Z_t)_{t\geq0}$ be an $\R^d$-valued Lévy process with Lévy triplet $(b,\bSigma,\nu)$, and  let $g\colon\R^d\to\R$ be a measurable, locally bounded and sub-multiplicative function. 
	Then, $\E[g(Z_t)]<\infty$ holds for all $t>0$ if and only if $\int_{\Vert z\Vert\geq 1} g(z)\nu(\d z)<\infty$.
\end{lemma}

In particular, the function $g(x)=(1\vee \Vert x\Vert)^p$ is sub-multiplicative for all $p>0$ (see Proposition 25.4 in \cite{sato99}) and thus $\E[\Vert Z_t\Vert^p]<\infty$ holds true for all $t>0$ as soon as $\int_{\Vert z \Vert \geq 1} \Vert z\Vert^p \nu(\d z)<\infty$ is fulfilled. 
We continue with the following result, which characterizes the invariant distribution of $\X$.

\begin{lemma}[cf.~Theorem 4.1 in \cite{sato84}, Proposition 2.2 in \cite{mas04}]\label{lemma: inv dist ou}
	Assume \ref{ass: ergodicity}. 
	Then, $\X$ has a unique invariant distribution $\mu$ which is infinitely divisible with characteristic triplet $(b_\mu,\C_\mu,\nu_\mu)$ where
	\begin{align*}
		b_\mu&\ = \A^{-1}b +\int \int_0^\infty \e^{-s\A}z\left( \mathds{1}_{\Vert z\Vert\leq 1}(\e^{-s\A}z)-\mathds{1}_{\Vert z\Vert\leq 1}(z)\right)\d s \nu (\d z),
		\\
		\C_\mu&\ = \int_0^\infty \e^{-s\A}\C \e^{-s\A^\top}\d s,
		\\
		\nu_\mu(B)&\ =\int_0^\infty \nu(\e^{s\A}B)\d s, \quad \forall B\in \mathcal{B}(\R^d),
		\\
		\e^{s\A}B&\coloneq \left\{y\in\R^d: y=\e^{s\A}x, x\in B \right\},\quad \forall B\in \mathcal{B}(\R^d).
	\end{align*}
\end{lemma}

Combining these results directly leads to the following corollary.

\begin{corollary}\label{cor: inv dens moments}
	Consider an $\R^d$-valued Lévy process $\Z=(Z_t)_{t\geq0}$ with Lévy triplet $(b,\bSigma,\nu)$, and assume \ref{ass: ergodicity}.
	Let $p\geq2$ be given, and suppose that $\int_{\Vert z \Vert \geq 1}\Vert z\Vert^p \nu(\d z)<\infty$. 
	Then, \[\int \Vert x\Vert^p \mu(\d x)<\infty.\] 
\end{corollary}
\begin{proof}
	By Lemmas \ref{lemma: levy proc moment} and  \ref{lemma: inv dist ou}, it suffices to show $\int_{\Vert z \Vert \geq 1} \Vert z\Vert^p \nu_\infty(\d z)<\infty$.
	It holds
	\begin{align*}
		\int_{\Vert z \Vert \geq 1} \Vert z\Vert^p \nu_\mu(\d z)\leq \int_0^\infty\int \Vert \e^{-s\A}z\Vert^p \nu(\d z)\d s\leq \int_0^\infty\Vert \e^{-s\A}\Vert_{2}^p\d s\int \Vert z\Vert^p \nu(\d z)<\infty,
	\end{align*} 
	since $\nu$ is a Lévy measure, $p\geq2$ and $\A\in M_+(\R^d)$.
\end{proof}

Another consequence of Lemma \ref{lemma: inv dist ou} is that $\kappa_{\min}>0$ follows from \ref{ass: ergodicity}, since then by assumption the Gaussian part of $\mu$ is nontrivial. Hence, the support of $\mu$ cannot be contained in a hyperplane of $\R^d,$ which would be the case if $\kappa_{\min}$ was equal to $0$.

\section{Proofs for Section \ref{sec: results}}\label{app: res proof}
\begin{proof}[Proof for Lemma \ref{lemma: L2 frobenius}]
	We adapt the proof of Lemma A.2 in \cite{belets18}; also cf.~the proof of Lemma 3 in \cite{gama19}.
	Define the functions $f$ and $g$ by the relations $g(\A)=\mathcal{L}_T(\A), f\equiv g+h$. 
	By Proposition \ref{prop: sorensen likelihood}, we have that 
	\begin{align*}
		\mathcal{L}_T(\A)&=\frac 1 T\int_0^T  ( \C^{-1} \A X_{s-})^\top \d X^{\mathsf{c}}_s+\frac{1}{2T} \int_0^T(\bSigma^{-1} \A X_{s-})^\top  \bSigma^{-1} \A X_{s-} \d s.
	\end{align*} 
	Note that, under $\P^{\boldsymbol{0}}$, $X_t^{\mathsf{c}}=\bSigma W_t$, where $W$ is a $\P^{\boldsymbol{0}}$-Wiener process.
	Additionally, by Girsanov's theorem,
	$$\tilde{W}_t=W_t+\bSigma^{-1} \A_0 \int_0^t X_s\d s$$
	is a $\P^{\A_0}$-Wiener process. 
	Hence, we can write
	\begin{align*}
		\mathcal{L}_T(\A)&=\frac 1 T\int_0^T  ( \bSigma^{-1} \A X_s)^\top \d \tilde{W}_s+\frac{1}{2T}
		\int_0^T(\bSigma^{-1} \A X_s)^\top  \bSigma^{-1} \A X_s \d s
		-\frac 1 T \int_0^T(\bSigma^{-1}\A X_s)^\top \bSigma^{-1} \A_0 X_s\d s \\
		&=\frac{1}{2T}\tr\left(2\bSigma^{-1}\A\int_0^T X_s\d \tilde{W}^\top_s+\bSigma^{-1}\A\int_0^T X_s X_s^\top\d s (\bSigma^{-1}(\A-2\A_0))^\top \right)
		\\
		&= \tr\left(\bSigma^{-1}\A\boldsymbol{\ep}^\top_T+\frac{1}{2}\bSigma^{-1}\A \hat{\C}_T(\bSigma^{-1}\A)^\top -\bSigma^{-1}\A \hat{\C}_T(\bSigma^{-1}\A_0)^\top \right)
		\\
		&= \tr\left(\bSigma^{-1}\A\boldsymbol{\ep}^\top_T+\frac{1}{2}\A \hat{\C}_T\A^\top\C^{-1} -\A \hat{\C}_T\A_0^\top\C^{-1} \right),
	\end{align*} 
	where $\boldsymbol{\ep}_T$ and $\hat{\C}_T$ are defined according to \eqref{def:ep} and \eqref{def:C}, respectively.
	The gradient is thus given as $(\bSigma^{-1})^\top\boldsymbol{\ep}_T+\C^{-1} (\A-\A_0)\hat{\C}_T$. 
	Since $f$ is convex, it follows that $\mathbf{0}$ is in the subdifferential of $f$ at $\hat{\A}$.
	The Moreau--Rockafellar theorem then gives that there exists $\mathbf B$ in the subdifferential of $h$ at $\hat{\A}$ such that $\mathbf{0}=(\bSigma^{-1})^\top\boldsymbol{\ep}_T+\C^{-1} (\hat{\A}-\A_0)\hat{\C}_T +\mathbf{B}$. 
	Additionally, $\mathbf{B}$ being in the subdifferential of $h$ at $\hat{\A}$ implies $\langle \mathbf{B}, \A-\hat{\A}\rangle_{2}\leq h(\A)-h(\hat{\A})$. 
	Consequently,
	\begin{align*}
		&\Vert\bSigma^{-1}(\hat{\A}-\A_0)X\|_{L^2}^2-\|\bSigma^{-1}(\A-\A_0) X\|_{L^2}^2+\|\bSigma^{-1}(\hat{\A}-\A)X\|_{L^2}^2
		\\
		&\hspace*{3em}=\langle(\hat{\A}-\A_0)^\top\C^{-1}(\hat{\A}-\A_0)-(\A-\A_0)^\top\C^{-1}(\A-\A_0)+(\hat{\A}-\A)^\top\C^{-1}(\hat{\A}-\A), \hat{\C}_T\rangle_{2}
		\\
		&\hspace*{3em}=2\langle (\hat{\A}-\A_0)^\top\C^{-1}(\hat{\A}-\A), \hat{\C}_T\rangle_{2}
		\\
		&\hspace*{3em}=2\langle(\hat{\A}-\A)^\top, \hat{\C}_T(\hat{\A}-\A_0)^\top\C^{-1}\rangle_{2}
		\\
		&\hspace*{3em}=2\langle(\A-\hat{\A}),(\bSigma^{-1})^\top\boldsymbol{\ep}_T+\mathbf{B}\rangle_{2}
		\\
		&\hspace*{3em}\leq 2(\langle\bSigma^{-1}(\A-\hat{\A}),\boldsymbol{\ep}_T\rangle_{2}+h(\A)-h(\hat{\A})).
	\end{align*}
\end{proof}

\section{Proofs for Section \ref{sec: ass conc}} \label{app: proof conc}
	\begin{proof}[Proof of Proposition \ref{thm: levy conc}]
	Let $u\in\R^d$ be given such that $\Vert u\Vert \leq 1$.
	Recall that, for $s>0$, $X_s$ is given explicitly as
	\begin{align*}
		X_s=\e^{-s\A}X_0+\int_0^s\e^{-(s-r)\A}\d Z_r.
	\end{align*}
This implies that 
	\begin{align*}
		(u^\top X_s)^2=(u^\top \e^{-s\A}X_0)^2+2 (u^\top \e^{-s\A}X_0) (u^\top  Y_s)+(u^\top Y_s)^2,
	\end{align*}
	where
	\[Y_s\coloneq \int_0^s \e^{-(s-r)\A}\d Z_r. \]
	Now, by the Lévy-- It\={o} decomposition, for all $s>0$,
	\begin{align*}
		Y_s&=\int_0^s \e^{-(s-r)\A}b\d r +\int_0^s  \e^{-(s-r)\A}\bSigma \d W_r+\int_0^s\int_{\vert z\vert\geq 1}  \e^{-(s-r)\A}zN(\d r, \d z)\\
		&\hspace*{3em}+\int_0^s\int_{\vert z\vert< 1} \e^{-(s-r)\A}z\tilde{N}(\d r, \d z),
	\end{align*}
	which allows us to apply  It\={o}'s formula (see e.g. Theorem 4.4.7 in \cite{applebaum09}).
	It gives
	\begin{align}\notag
		&(u^\top Y_s)^2\\&=2\int_0^s (u^\top Y_{r-}) u^\top\e^{-(s-r)\A}b\d r+ 2 \int_0^s(u^\top Y_{r-}) \notag u^\top\e^{-(s-r)\A} \bSigma\d W_r+\int_0^s u^\top\e^{-(s-r)\A} \C\e^{-(s-r)\A^\top}u\d r
		\\&\notag\quad+\int_0^s \int_{\Vert z \Vert \geq1}(u^\top(Y_{r-}+\e^{-(s-r)\A}z))^2-( u^\top Y_{r-})^2N(\d r,\d z)
		\\&\notag\quad+\int_0^s \int_{\Vert z \Vert <1}(u^\top( Y_{r-}+\e^{-(s-r)\A}z))^2-( u^\top Y_{r-})^2\tilde{N}(\d r,\d z)
		\\&\notag\quad +\int_0^s\int_{\Vert z\Vert <1} (u^\top(Y_{r-}+\e^{-(s-r)\A}z))^2-( u^\top Y_{r-})^2-2(u^\top Y_{r-})u^\top \e^{-(s-r)\A}z\nu(\d z)\d r
		\\
		&\notag=2\int_0^s (u^\top Y_{r-}) u^\top\e^{-(s-r)\A}b^*\d r+ 2 \int_0^s(u^\top Y_{r-}) u^\top\e^{-(s-r)\A} \bSigma\d W_r+\int_0^s u^\top\e^{-(s-r)\A} \C\e^{-(s-r)\A^\top}u\d r
		\\&\quad+\int_0^s\int2(u^\top Y_{r-})u^\top \e^{-(s-r)\A}z+(u^\top \e^{-(s-r)\A}z)^2 \tilde{N}(\d r,\d z)
		+\int_0^s\int(u^\top \e^{-(s-r)\A}z)^2\nu(\d z)\d r \label{eq: ito},
	\end{align}
	where $b^*=b+\int_{\Vert z \Vert \geq 1}z\nu(\d z)$. 
	Stationarity of $\X$ implies, for any $s\geq 0$,
	\begin{align*}
		(u^\top X_s)^2-\int (u^\top x)^2 \mu(\d x)&=(u^\top X_s)^2-\E[(u^\top X_s)^2]\\
		&=u^\top \e^{-s\A}(X_0X_0^\top-\E[X_0X_0^\top])\e^{-s\A^\top}u\\
		&\hspace*{3em}+2 u^\top \e^{-s\A}(X_0Y_s^\top-\E[X_0Y_s^\top])u+(u^\top Y_s)^2-\E[(u^\top Y_s)^2],
	\end{align*}
	and \eqref{eq: ito} gives
	\begin{align*}
		(u^\top Y_s)^2-\E[(u^\top Y_s)^2]
		&=2\int_0^s u^\top(Y_{r-}-\E[Y_{r-}]) u^\top\e^{-(s-r)\A}b^*\d r
		+2 \int_0^s(u^\top Y_{r-}) u^\top\e^{-(s-r)\A} \bSigma\d W_r
		\\&\hspace*{3em}	+\int_0^s\int 2u^\top Y_{r-}u^\top \e^{-(s-r)\A}z+(u^\top \e^{-(s-r)\A}z)^2 \tilde{N}(\d r,\d z).
	\end{align*}
Additionally, the independence of $X_0$ and $\Z$ leads to
	\begin{align*}
		\E[X_0Y_s^\top]&=\E[X_0](\A^{-1}(\operatorname{Id}-\e^{-s\A})b^*)^\top.
	\end{align*}
	Hence,
	\begin{align*}
		&\int_0^t (u^\top X_s)^2-\int (u^\top x)^2 \mu(\d x) \d s
		\\
		&=\int_0^t u^\top \e^{-s\A}(X_0X_0^\top-\E[X_0X_0^\top])\e^{-s\A^\top}u+2 u^\top \e^{-s\A}(X_0Y_s^\top-\E[X_0](\A^{-1}(\operatorname{Id}-\e^{-s\A})b^*)^\top) u
		\\&\quad+\int_0^s2(u^\top Y_{r-}) u^\top\e^{-(s-r)\A} \bSigma\d W_r+\int_0^s\int2u^\top  Y_{r-}u^\top \e^{-(s-r)\A}z+(u^\top \e^{-(s-r)\A}z)^2 \tilde{N}(\d r,\d z) \d s
		\\&= A^1_t+A^2_t+M^c_t+M^d_t,
	\end{align*}
	where
	\begin{align*}
		A^1_t&= \int_0^t u^\top \e^{-s\A}(X_0X_0^\top-\E[X_0X_0^\top])\e^{-s\A^\top}u\d s,
		\\
		A^2_t&=\int_0^t 2 u^\top \e^{-s\A}(X_0Y_s^\top-\E[X_0](\A^{-1}(\operatorname{Id}-\e^{-s\A})b^*)^\top) u\d s,
		\\
		M^c_t&=\int_0^t  \int_0^s2(u^\top  Y_{r-}) u^\top\e^{-(s-r)\A} \bSigma\d W_r\d s,
		\\
		M^d_t&= \int_0^t \int_0^s\int2(u^\top Y_{r-})u^\top \e^{-(s-r)\A}z+(u^\top \e^{-(s-r)\A}z)^2 \tilde{N}(\d r,\d z) \d s.
	\end{align*}
	For the following bound, first note that, since $\A$ is diagonalizable, there exists some matrix $V$ such that 
	\[
	\forall s\in\R, \qquad \Vert \e^{s\A}\Vert^2_{2}\leq \alpha^2 \e^{2\max_i(s\Re(\lambda_i))},
	\] 
	where $\alpha=\Vert V\Vert_{2}\Vert V^{-1}\Vert_{2}\sqrt{d}>0$ and $\lambda_i$, $i=1,\ldots,d$, are the eigenvalues of $\A$. 
	In particular, for $s>0$ it holds $\Vert \e^{-s\A}\Vert^2_{2}\leq \alpha^2 \e^{-2s\beta}$, where $\beta=\min_i\Re(\lambda_i)>0$ by \ref{ass: ergodicity}. 
	The It\={o} isometry thus implies 
	\begin{align}
		\E[\Vert Y_s\Vert^2]&=\E\left[\int_0^s \Vert \e^{-(s-r)\A}\bSigma\Vert^2_{2}\d r\right]+\E\left[\int_0^s\int \Vert\e^{-(s-r)\A}z\Vert^2N(\d r, \d z) \right]+\Vert \A^{-1}(\operatorname{Id}-\e^{-s\A})b^*\Vert^2\notag
		\\
		&\leq \alpha^2 \Vert \bSigma\Vert^2_{2} \int_0^s  \e^{-2(s-r)\beta}\d r+\alpha^2 \int_0^s \e^{-2(s-r)\beta}\int \Vert z\Vert^2\nu(\d z) \d r+\alpha^2\Vert \A^{-1}\Vert_{2}^2\Vert b^*\Vert^2\notag
		\\
		&\leq (\alpha^2/(2\beta))( \Vert \bSigma\Vert^2_F+\int \Vert z\Vert^2\nu(\d z))+\alpha^2\Vert \A^{-1}\Vert_{2}^2\Vert b^*\Vert^2\eqqcolon c_1.\label{eq: martingale bound}
	\end{align}
	Now, \eqref{eq: martingale bound} yields
	\begin{align*}
		\E[\vert A^1_t\vert]&\leq  \int_0^t\alpha^2\e^{-2\beta s}  \E[\Vert X_0X_0^\top-\E[X_0X_0^\top]\Vert_2]\d s
		\\
		&\leq \alpha^2/(2\beta)\E[\Vert X_0X_0^\top-\E[X_0X_0^\top]\Vert_2]\eqqcolon c_2
	\end{align*}
	and
	\begin{align*}
		\E[\vert A^2_t\vert]&\leq 2\int_0^t  \alpha\e^{-s\beta}(\E[\Vert X_0\Vert]\sqrt{c_1}+\alpha\Vert \A^{-1}\Vert_{2}\Vert b^*\Vert\E[\Vert X_0\Vert]) \d s
		\\
		&\leq 2(\alpha/\beta)\E[\Vert X_0\Vert](\sqrt{c_1}+\alpha\Vert \A^{-1}\Vert_{2}\Vert b^*\Vert)  \eqqcolon c_3.
		\end{align*} 
	Turning our attention to $M^c_t$ and $M^d_t$, Fubini's theorem for stochastic integrals (see, e.g., Theorem 65 in \cite{protter2004}) gives
	\begin{align*}
		M^c_t&=\int_0^t2(u^\top Y_{r-}) u^\top\e^{r\A}  \int_r^t\e^{-s\A} \d s\bSigma\d W_r
		\\
		&=\int_0^t2(u^\top Y_{r-}) u^\top (\operatorname{Id}-\e^{-(t-r)\A})\A^{-1}\bSigma\d W_r	
	\end{align*}
	and
	\begin{align*}
		M^d_t&= \int_0^t\int \int_r^t\left(2u^\top Y_{r-}u^\top \e^{-(s-r)\A}z+(u^\top \e^{-(s-r)\A}z)^2\right)\d s \tilde{N}(\d r,\d z) 
		\\
		&=\int_0^t\int\left(2u^\top Y_{r-} u^\top (\operatorname{Id}-\e^{-(t-r)\A})\A^{-1}z+ \int_r^t(u^\top \e^{-(s-r)\A}z)^2\d s\right) \tilde{N}(\d r,\d z).
	\end{align*}
	This, together with \eqref{eq: martingale bound} and the  It\={o} isometry, implies
	\begin{align*}
		\E[\vert M^c_t\vert^2]&\leq4c_1\Vert \A^{-1}\bSigma\Vert^2_{2} \int_0^t\left( 2+ 2\alpha^2\e^{-2(t-r)\beta}\right)\d r
		\\
		&\leq 4c_1\Vert \A^{-1}\bSigma\Vert^2_{2} \left(2t+ \frac{\alpha^2}{\beta}\right).
	\end{align*}
	Similarly, we obtain
	\begin{align*}
		\E\left[\vert M^d_t\vert^2\right]&\leq 4c_1\Vert\A^{-1}\Vert^2_{2} \left(\int \Vert z\Vert^2\nu(\d z) \int_0^t\Vert (\operatorname{Id}-\e^{-(t-r)\A})\Vert_2^2\d r\right)\\
		&\hspace*{3em}+ \left(\int \Vert z\Vert^4 \nu(\d z )\int_0^t\left(\int_r^t \Vert\e^{-(s-r)\A}\Vert_{2}^2\d s\right)^2\d r\right)
		\\
		&\leq 4c_1\Vert\A^{-1}\Vert_2^2 \left(\int \Vert z\Vert^2\nu(\d z) \left(2t+ \frac{\alpha^2}{\beta}\right)\right)+ \left(\int \Vert z\Vert^4 \nu(\d z)\frac{\alpha^4}{4\beta^2}t \right),
	\end{align*}
	which is finite by the assumption of $Z_1$ admitting a fourth moment.
	Markov's inequality now implies for any $r>0$
	\begin{align*}
		&\P\left(\vert u^\top (\hat{\C}_T-\C_\infty)u\vert \geq r \right)
		\\
		&\hspace*{3em}\leq \P\left(\vert A^1_t\vert \geq \frac{tr}{4} \right)+ \P\left(\vert A^2_t\vert \geq \frac{tr}{4} \right)+\P\left(\vert M^c_t\vert \geq \frac{tr}{4}  \right)+\P\left(\vert M^d_t\vert \geq \frac{tr}{4}  \right)
		\\
		&\hspace*{3em}\leq \frac{4(c_2+c_3)}{tr}+\frac{ 64c_1\Vert \A^{-1}\bSigma\Vert^2_{2} \left(2t+ \frac{\alpha^2}{\beta}\right)}{(tr)^2}\\
		&\hspace*{6em}+\frac{64c_1\Vert\A^{-1}\Vert_{2}^2 \left(\int \Vert z\Vert^2\nu(\d z) \left(2t+ \frac{\alpha^2}{\beta}\right)\right)+ 4\left(\int \Vert z\Vert^4 \nu(\d z)\frac{\alpha^4}{\beta^2}t \right)}{(tr)^2}\\
		&\hspace*{3em}\leq \frac{4(c_2+c_3)+128c_1\Vert \A^{-1}\Vert^2_{2}(\Vert\bSigma\Vert^2_{2}+\int \Vert z\Vert^2 \nu(\d z))+(2\alpha^2/\beta)^2\int \Vert z\Vert^4 \nu(\d z)}{t(r\land r^2)}
		\\&\hspace*{6em}+ \frac{64c_1\alpha^2\Vert \A^{-1}\Vert^2_{2}(\Vert\bSigma\Vert^2_{2}+\int \Vert z\Vert^2 \nu(\d z))}{\beta (tr)^2},
	\end{align*}
	which concludes the proof.
\end{proof}
\printbibliography

\end{document}